\documentclass[reqno,12pt,twosided,a4paper]{amsart}
\usepackage{amssymb,amscd,stmaryrd}
\usepackage{amsmath}
\usepackage{amsfonts}
\usepackage[top=3cm,bottom=3cm,left=2.5cm,right=2.5cm,footskip=17mm]{geometry}
\usepackage[mathcal]{eucal}
\usepackage{array,float}
\usepackage{enumitem}
\usepackage[OT2,T1]{fontenc}
\usepackage[russian,USenglish]{babel}
\usepackage{chngcntr}
\usepackage{apptools}
\AtAppendix{\counterwithin{theorem}{section}}

\usepackage{amsmath}
\usepackage{amssymb, xy}
\input xy
\xyoption{all}
\usepackage{pdflscape}
\usepackage{hyperref}
\usepackage[draft]{graphicx}
\usepackage{xcolor}
\usepackage{color}

\numberwithin{equation}{section}

\newtheorem{theorem}{Theorem}[section]
\newtheorem{lemma}[theorem]{Lemma}
\newtheorem{proposition}[theorem]{Proposition}

\theoremstyle{definition}
\newtheorem{definition}[theorem]{Definition}

\newtheorem{remark}[theorem]{Remark}

\newcommand{\ot}{\otimes}

\newcommand{\N}{\mathbb{N}}

\newcommand{\one}{\textbf{1}}

\newcommand{\Z}{\mathbb{Z}}
\newcommand{\Ga}{\Gamma}

\newcommand{\A}{\mathbb{A}}

\newcommand{\Ker}{\text{Ker}}
\newcommand{\End}{\text{End}}
\newcommand{\Hom}{\text{Hom}}
\newcommand{\M}{\text{M}}
\newcommand{\parti}{\vdash}
\newcommand{\la}{\lambda}

\renewcommand{\S}{\mathbb{S}}
\newcommand{\GL}{\text{GL}}
\newcommand{\R}{\text{R}}
\newcommand{\Ind}{\text{Ind}}
\newcommand{\spanG}{span_{\Theta}}
\newcommand{\Q}{\mathbb{Q}}
\newcommand{\tpi}{\frac{1}{2\pi i}}
\newcommand{\Cc}{\mathbb{C}}
\newcommand{\Res}{\text{Res}}

\title{Invariant rings and representations of the symmetric groups}
\author{Ehud Meir}
\begin{document}
\maketitle
\begin{abstract}
In this paper we study invariant rings arising in the study of finite dimensional algebraic structures. The rings we encounter are graded rings of the form $K[U]^{\Ga}$ where $\Ga$ is a product of general linear groups over a field $K$ of characteristic zero, and $U$ is a finite dimensional rational representation of $\Ga$. We will calculate the Hilbert series of such rings using the representation theory of the symmetric groups and Schur-Weyl duality. 
We focus on the case where $U=\End(W^{\oplus k})$ and $\Ga = \GL(W)$ and on the case where $U=\End(V\ot W)$ and $\Ga = \GL(V)\times \GL(W)$, though the methods introduced here can also be applied in more general framework. For the two aforementioned cases we calculate the Hilbert function of the ring of invariants in terms of Littlewood-Richardson and Kronecker coefficients. When the vector spaces are of dimension 2 we also give an explicit calculation of this Hilbert series.
\end{abstract}
 
\begin{section}{Introduction}
Let $K$ be a field of characteristic zero.
Let $W$ be a finite dimensional algebraic structure over $K$. This can be, for example, an algebra, a Hopf algebra, a comodule algebra et cetera (see \cite{meir3} for a discussion about finite dimensional algebraic structures). Such algebraic structures are given by structure tensors, which are linear maps between tensor powers of $W$. 
For example, a multiplication is given by a map $m:W\ot W\to W$, and a comultiplication by a map $\Delta:W\to W\ot W$. We then understand the algebraic structure as the tuple formed by all structure tensors $(m,\Delta,\ldots)$.

Geometric invariant theory provides a natural tool to study such algebraic structures. This was carried out in \cite{meir1} and \cite{DKS} for finite dimensional semisimple Hopf algebras. The idea is the following: by fixing some discrete invariants such as the dimension of the Hopf algebra in \cite{DKS} or the dimension of the irreducible representations in \cite{meir1}, and by fixing a basis for the Hopf algebra, the Hopf algebra can be described using structure constants. In this way a Hopf algebra can be seen as a point in a certain affine space $\A^N$ (where $N$ is the total number of structure constants involved). Not all points in $\A^N$, however, define Hopf algebras. The subset of points which do define Hopf algebras is an affine sub-variety $X\subseteq \A^N$. 

The structure constants, while containing all the information about the Hopf algebra, are not \textit{invariants}. They depend on the particular choice of basis for the Hopf algebra or for the irreducible representation of it. The affine space $\A^N$ is equipped with an action of a reductive algebraic group $\Ga$ which stabilizes $X$ such that two points in $X$ define isomorphic Hopf algebras if and only if they lie in the same orbit of $\Ga$. For this reason, the ring of invariants $K[X]^{\Ga}$ comes into play here. Indeed, it is known that all the orbits of the action of $\Ga$ on $X$ are closed. Geometric Invariant Theory (GIT) tells us that in this case the quotient set $X/\Ga$ is again an affine variety, and that $K[X/\Ga]\cong K[X]^{\Ga}.$

Since the group $\Ga$ is reductive, and since restriction of polynomial functions from $\A^N$ to $X$ is surjective, the short exact sequence $$0\to I=\Ker(\text{Res})\to K[\A^N]\to K[X]\to 0$$ gives rise to a short exact sequence $$0\to I^{\Ga}\to K[\A^N]^{\Ga}\to K[X]^{\Ga}\to 0.$$ In other words, $K[X]^{\Ga}\cong K[\A^N]^{\Ga}/I^{\Ga}$. This means that in order to study $K[X]^{\Ga}$ we should study $K[\A^N]^{\Ga}$ and $I^{\Ga}$.

Using Schur-Weyl duality, a set of generators (and in fact, a linear spanning set) for $K[\A^N]^{\Ga}$ was described in \cite{DKS} and in \cite{meir1}. The use of Schur-Weyl duality to describe the invariants stems from the work of Procesi \cite{Procesi} who studied tuples of linear endomorphisms of a finite dimensional vector space.  Understanding the relations between the generators arising from the Schur-Weyl duality is more difficult.
In \cite{meir2} a similar GIT quotient was studied for two-cocycles over an arbitrary finite dimensional Hopf algebra. A description of all relations among these generators was also given. The resulting presentation, however, still has infinitely many generators and infinitely many relations. A finite presentation for the ring of invariants was then given in specific cases. 

The most difficult part of the relations among the generators of the ring of invariants are the relations arising from the dimensions of the associated vector spaces. In this paper we will study the Hilbert function and Hilbert series of the rings of invariants using the representation theory of the symmetric groups. We shall do so for two different invariant theory problems. The first will be the invariants for an endomorphism of a tensor product of vector spaces. This problem arises in the study of finite dimensional semisimple Hopf algebras. The second one will be the invariants of a tuple of endomorphisms of a vector space, a problem that was studied by Procesi in \cite{Procesi}. 

In \cite{meir1} the affine space $\A^N$ mentioned above was $\oplus_{i,j}\End(V_i\ot W_j)$ where $(V_i)_i$ are the irreducible representations of the Hopf algebra $H$ and $(W_j)_j$ are the irreducible representations of the Hopf algebra $H^*$. The affine group was $\prod_i\GL(V_i)\times \prod_j \GL(W_j)$. 
We will relax this problem here, and study the invariant ring $$K[\End(V\ot W)]^{\GL(V)\times \GL(W)}.$$
In Section \ref{sec:tensorproducts} we will prove the following:
\begin{theorem}[see Theorem \ref{thm:maintensorproducts}] Assume that $\dim(V)=d_1$ and $\dim(W)=d_2$. The dimension of the $n$-th homogeneous component of $K[\End(V\ot W)]^{\GL(V)\times \GL(W)}$ is 
$$\sum_{\substack{\la,\mu,\nu\parti n\\ r(\la)\leq \dim(V) \\ r(\mu)\leq\dim(W)}} g(\la,\mu,\nu)^2$$
where $g(\la,\mu,\nu)$ are the Kronecker coefficients of $S_n$, and the sum is taken over all partitions $\la$ of $n$ with at most $d_1$ rows and all partitions $\mu$ of $n$ with at most $d_2$ rows.
\end{theorem}
In the specific case where $d_1=d_2=2$, we have the following more concrete calculation:
\begin{theorem} When $d_1=d_2=2$, the Hilbert series of $K[\End(V\ot W)]^{\GL(V)\times GL(W)}$ is the following rational function:
$$\frac{1 - x^2 - x^3 + 2x^4 + 2x^5 + 2x^6 - x^7 - x^8 + x^{10}}{(1-x)(1-x^2)^4(1-x^3)^3(1-x^4)^2}$$
\end{theorem}

Studying rings of invariants using the combinatorics of the symmetric groups was initiated in \cite{Procesi}.
Procesi studied $k$-tuples of $d\times d$ matrices under the action of conjugation by the same invertible matrix. 
In the language of algebraic structures, this can be understood as a vector space of dimension $d$ equipped with $k$ linear endomorphisms. If we denote the matrices by $(M_1,\ldots M_k)$ then a generating set of invariants is given by 
$$\{tr(M_{i_1}M_{i_2}\cdots M_{i_l})| 1\leq i_1,i_2,\ldots i_l\leq k,l\in \N\}.$$ Procesi also proved that all the relations between these invariants are derived from the cyclic property of the trace and the Cayley-Hamilton Theorem. 

In this paper we will give a concrete description of the Hilbert function of the invariant ring $K[\End(W)^{\oplus k}]^{\GL(W)}$ using the Littlewood-Richardson coefficients. We have the following result:
\begin{theorem}[see Theorem \ref{thm:maintuples}]
Assume that $\dim(W)=d$. The $n$-th term of the Hilbert function of $K[\End(W)^{\oplus k}]^{\GL(W)}$
is $$\sum_{n_1+n_2+\cdots +n_k=n}\sum_{\substack{\la\parti n \\ r(\la)\leq d}}\sum_{\la_1\parti n_1}\cdots\sum_{\la_k\parti n_k}(c_{(\la_i)}^{\la})^2$$
where for $\la_i\parti n_i$ the iterated Littlewood-Richardson coefficient $c_{(\la_i)}^{\la}$ is given in Definition \ref{def:iteratedLR}.
\end{theorem}
Again, in case $d=2$ we get a more concrete description of the Hilbert series:
\begin{theorem}\label{thm:maintuples2} In case $d=2$ the Hilbert series of $K[\End(W)^{\oplus k}]^{\GL(W)}$ is 
$$\frac{1}{(1-x^2)^k(1-x)^{2k}}\cdot \Bigg[\sum_{i=0}^{k-1} \binom{k-1}{i}\binom{2k-2-i}{k-1}\frac{x^{2k-2-2i}}{(1-x^2)^{k-1-i}}-$$
$$\binom{k}{i}\binom{2k-2-i}{k-1}\frac{x^{2k-1-2i}}{(1-x^2)^{k-1-i}}\Bigg].$$
\end{theorem}
This paper is organised as follows:
In Section \ref{sec:prelim} we recall some well known results about Hilbert functions, Hilbert series and representations of the symmetric groups. We will also recall Zelevinsky's approach to the representation theory of the symmetric groups, using the PSH-algebra $Zel$. In Section \ref{sec:star} we will explain how the Kronecker product of representations can be understood in terms of Zelevinsky's algebra $Zel$.
In Section \ref{sec:tensorproducts} we will apply the results from previous sections to the study the Hilbert function of the invariant ring $K[\End(V\ot W)]^{\GL(V)\times \GL(W)}$ in terms of the Kronecker coefficients, and calculate the Hilbert series explicitly in case $\dim(V)=\dim(W)=2$. 
In Section \ref{sec:tuples} we will study the Hilbert function of the invariant ring $K[\End(W)^{\oplus k}]^{\GL(W)}$ in terms of the Littlewood-Richardson coefficients, and calculate explicitly the Hilbert series in case $\dim(W)=2$. 
\end{section}

\begin{section}{Preliminaries and notations}\label{sec:prelim}
Throughout this paper we will work over a field $K$ of characteristic zero. All vector spaces we will consider will be over $K$, all tensor products will be taken over $K$ (unless otherwise specified), and all dimensions we will consider will be dimensions over $K$. 
We recall that for a graded algebra \begin{equation}A=\bigoplus_{n\geq 0} A_n\end{equation} in which all the homogeneous components $A_n$ are finite dimensional the \textit{Hilbert function} is given by 
\begin{equation}f(n) = \dim A_n \end{equation}
and the \textit{Hilbert series} is 
\begin{equation}
H_A(x) = \sum_n f(n)x^n = \sum_n\dim A_nx^n\in \Z[[x]].
\end{equation}
If $A$ is a Noetherian commutative ring its Hilbert series is in fact a rational function (see Chapter 11 of \cite{AM}).
If $m\in \N$ is a positive integer then \begin{equation}H_A(x)(1-x^m) = \sum_n (f(n)-f(n-m))x^n\end{equation} where it is understood that $f(n)=0$ for $n<0$. 
This equation will be used later when calculating recursive relations the Hilbert function satisfies.
\subsection{The representation theory of the symmetric group} 
We will follow here the approach of Zelevinsky to the representation theory of the symmetric groups from \cite{Zelevinsky}. The idea is to study the representation theory of \textit{all} the symmetric groups together, by combining them into one Hopf algebra over $\Z$.
For a finite group $G$, we write $\R(G)$ for the Grothendieck group of the category of complex representation of $G$. This is a free abelian group which has a canonical basis given by the irreducible representations of $G$. For any representation $V$ of $G$ we write $[V]$ for the isomorphism class of $V$ inside $\R(G)$.
We will use here freely the identification
$$\R(G)\ot_{\Z} \R(H)= \R(G\times H)$$
\begin{equation}[V]\ot [W]\mapsto [V\ot W]\end{equation} where $G$ and $H$ are two finite groups.

Let \begin{equation}Zel=\oplus_{n\geq 0} \R(S_n).\end{equation}
This is an abelian group which has an additional and much richer structure of a positive self-adjoint Hopf algebra (or PSH-algebra) which we will describe now.

The multiplication in $Zel$ is graded and is given by the formula \begin{equation}[V]\cdot [W] = [\Ind_{S_n\times S_m}^{S_{n+m}} V\ot W]\end{equation} where $[V]\in \R(S_n)$ and $[W]\in \R(S_m)$. 
The unit is the isomorphism class of the trivial representation of the symmetric group $S_0$. 
The comultiplication is given by the formula \begin{equation}\Delta([V]) = \sum_{a+b=n} [\text{Res}^{S_n}_{S_a\times S_b}V]\in\bigoplus _{a+b=n}\big( \R(S_a)\ot_{\Z} \R(S_b)\big)\end{equation}
for $V\in \R(S_n)$. 
The counit is given by \begin{equation}\epsilon([V]) = 0\text{ for } V\in \R(S_n), \, n>0 \text{ and }\epsilon(1)=1.\end{equation}
The pairing of characters gives us an inner product on $\R(S_n)$ for every $n$. 
The term ``self adjoint'' refers to the fact that by Frobenius reciprocity the multiplication is the adjoint operator to the comultiplication with respect to the inner product. 
The term ``positive'' refers to the fact that all the structure constants for the Hopf algebra operations with respect to the basis given by the irreducible representations of $S_n$ are positive. 

Zelevinsky proved in \cite{Zelevinsky} that the PSH-algebra $Zel$ is isomorphic to the polynomial algebra $\Z[x_1,x_2,\ldots]$, where the comultiplication is given by 
\begin{equation}\Delta(x_n) = \sum_{a+b=n} x_a\ot x_b.\end{equation} The element $x_n$ has degree $n$ and corresponds to the trivial representation of $S_n$.
He also proved that every PSH-algebra is isomorphic to a tensor product of copies of this PSH-algebra after rescaling of the degree (see Chapter I.2 in \cite{Zelevinsky}).

If $n$ is a non-negative integer a \textit{partition} of $n$ is a sequence $\lambda = (l_1,\ldots, l_r)$ of non-negative integers such that $l_1\geq l_2\geq \cdots\geq l_r>0$ and $\sum_i \la_i = n$. 
We call $r$ the \textit{length} of $\lambda$ and write $r(\lambda)=r$. We write $\la\parti n$ to indicate that $\la$ is a partition of $n$, and for any natural number $k$ we write \begin{equation}
P_k(n):= \{\la| \la\parti n\text{ and } r(\la)\leq k\}.\end{equation} 
The irreducible representations of $S_n$ are in one-to-one correspondence with partitions of $n$. 
To the partition $\lambda$ we assign the \textit{Specht module} $\S_{\lambda}$ (see Chapter 2 of \cite{Sagan}). Zelevinsky gave a concrete description of $[\S_{\lambda}]$ as an element of the algebra $Zel=\Z[x_1,x_2,\ldots]$. 
In Chapter II.6 of \cite{Zelevinsky}, he showed that if $\lambda = (l_1,l_2,\ldots l_r)$ then \begin{equation}[\S_{\lambda}] = det(x_{l_i+j-i})_{i,j}.\end{equation}
This implies in particular that if $\lambda = (a,n-a)$ with $a\geq n-a$ then 
\begin{equation}\label{eq:twoparts}[\S_{\lambda}] = x_ax_{n-a} - x_{a+1}x_{n-a-1}\end{equation} and for $\lambda = (n)$ we have 
\begin{equation}[V_{(n)}]= x_n.\end{equation} This is consistent with the previous assertion, that $x_n$ corresponds to the trivial representation of $S_n$. 

The structure constants for the algebra $Zel$ are given by the Littlewood-Richardson coefficients (see also Chapter 4.9. of \cite{Sagan}). For $\lambda\parti n$ and  $\mu\parti n$ we have
\begin{equation}[\S_{\lambda}]\cdot [\S_{\mu}] = \sum_{\nu\parti n+m} c_{\lambda,\mu}^{\nu}[\S_{\nu}].\end{equation}
Another important property of the Specht Modules is that they are all self-dual. Indeed, since it is known that all the Specht modules are defined over the field of rational numbers (Chapter 2 of \cite{Sagan}), it holds that their all their character values are rational. But this already implies that they are self-dual, since the character of the dual representation is given by the complex conjugation of the character of the representation.
\subsection{Schur-Weyl duality}
The connection between invariants with respect to general linear groups and representations of the symmetric groups is given by Schur-Weyl duality. We recall here the details. 
For a finite dimensional vector space $V$ and $n\in \N$, $V^{\ot n}$ is a representation of $S_n$ in a natural way. A permutation $\sigma\in S_n$ acts via the formula
\begin{equation}\sigma \cdot (v_1\ot\cdots\ot v_n) = v_{\sigma^{-1}(1)}\ot\cdots\ot v_{\sigma^{-1}(n)}.\end{equation}
We denote the resulting linear map $V^{\ot n}\to V^{\ot n}$ by $L_{\sigma}\in \End(V^{\ot n})$. This map commutes with the natural diagonal action of $\GL(V)$. Schur-Weyl duality can be phrased as the following statement (see the discussion in I.1 and Theorem 4.3. in \cite{Procesi}):
\begin{theorem}[Schur-Weyl duality]
\begin{enumerate} 
\item The linear map 
$$\Phi_V:KS_n\to (\End(V)^{\ot n})^{\GL(V)}$$ $$\sigma\mapsto L_{\sigma}$$ is a surjective ring homomorphism.
\item If we write the Wedderburn decomposition of the group algebra of $S_n$, $$KS_n = \bigoplus_{\la\parti n} \End(\S_{\la}),$$ where $\S_{\la}$ is the Specht module corresponding to the partition $\la$, then the kernel of $\Phi_V$ is $$\bigoplus_{\substack{\la\parti n \\ r(\la)> \dim(V)}}\End(\S_{\la}).$$
As a result, we have an isomorphism of algebras $$(\End(V)^{\ot m})^{\GL(V)}\cong \bigoplus_{\la\in P_{\dim(V)}(n)} \End(\S_{\la}).$$
\end{enumerate}
\end{theorem}
\begin{remark} 
Another way of understanding the kernel of $\Phi_V$ is the following: 
If $\dim(V)\geq n$ then $\Phi_V$ is injective, and if $\dim(V)<n$ then $\Ker(\Phi_V)$ is the two-sided ideal of $KS_n$ generated by the idempotent \begin{equation}
\frac{1}{(d+1)!}\sum_{\sigma\in S_{d+1}}(-1)^{\sigma}\sigma\end{equation}
where $d=\dim(V)$ (See Proposition 4.1 in \cite{meir2}). 
\end{remark}

The following lemma will be useful when calculating invariants:
\begin{lemma}
The map $\Phi_V:KS_n\to \End(V^{\ot n})$ is $S_n$-equivariant, where $S_n$ acts on $KS_n$ by conjugation, and on $\End(V^{\ot n})$ by conjugation. 
\end{lemma}
\begin{proof}
This is direct, since $\Phi_V$ is an algebra homomorphism, and the action of $\sigma\in S_n$ on $\End(V^{\ot n})$ is given by conjugation with $L_{\sigma}$.
\end{proof}
Schur-Weyl duality reduces the calculation of invariants of general linear groups to calculations in the representation theory of the symmetric group. In this paper we will encounter several different reductive groups. We will use the following lemma frequently, following Section 4 in \cite{meir2}.
\begin{lemma}\label{lem:inv-coinv}
Let $\Ga$ be a reductive group, and let $U$ be a rational representation of $\Ga$. 
The ring of invariants $K[U]^{\Ga}$ is then graded and we have a natural isomorphism
$$(K[U]^{\Ga})_n \cong (((U^*)^{\ot n})^{\Ga})_{S_n}$$ where $S_n$ acts on $(U^*)^{\ot n}$ as in the Schur-Weyl duality and $(-)_{S_n}$ are the $S_n$-coinvariants. 
\end{lemma}
\begin{proof}
We start by considering the bigger graded ring $K[U]= \oplus_{n\geq 0} K[U]_n$. This is a quotient of the graded tensor algebra 
\begin{equation}T(U^*) = \oplus_{n\geq 0} (U^*)^{\ot n}.\end{equation} For every $n$ we have an exact sequence of $\Ga$-representations 
\begin{equation}\bigoplus_{\sigma\in S_n}(U^*)^{\ot n}e_{\sigma}\stackrel{i}{\to} (U^*)^{\ot n}\stackrel{\pi}{\to} K[U]_n\to 0\end{equation}
where the map $i$ is given by $i(xe_{\sigma}) = x - \sigma(x)$ for $x\in (U^*)^{\ot n}$.
Taking $\Ga$-invariants, the fact that $\Ga$ is a reductive group gives us an exact sequence  
\begin{equation}\bigoplus_{\sigma\in S_n}((U^*)^{\ot n}e_{\sigma})^{\Ga}\stackrel{i}{\to} ((U^*)^{\ot n})^{\Ga}\stackrel{\pi}{\to} (K[U])^{\Ga}_n\to 0\end{equation}
which can be re-written as 
\begin{equation}\bigoplus_{\sigma\in S_n}((U^*)^{\ot n})^{\Ga}e_{\sigma}\stackrel{i}{\to} ((U^*)^{\ot n})^{\Ga}\stackrel{\pi}{\to} (K[U])^{\Ga}_n\to 0\end{equation}
This short exact sequence gives us the result:
\begin{equation}(K[U]^{\Ga})_n\cong (((U^*)^{\ot n})^{\Ga})_{S_n}.\end{equation}
\end{proof}
The following lemma about restriction of representations will be needed in Section \ref{sec:tuples}:
\begin{lemma}\label{lem:RepsRes}
Let $G$ be a finite group, and let $H$ be a subgroup of $G$.
For every $G$-representation $V$ we have a canonical isomorphisms
$$\text{Ind}_H^G\text{Res}^G_H V\cong KG/H\ot V$$
\end{lemma}
\begin{proof}
The isomorphism is given explicitly by 
$$\phi:\text{Ind}_H^G\text{Res}^G_H V\to KG/H\ot V$$
$$g\ot v\mapsto gH\ot gv$$
A direct verification shows that this is well defined and an isomorphism.
\end{proof}

We finish this section with the theorem of residues from complex analysis (see Theorem 19 in Section 4.5 of \cite{Ahlfors}). This will become useful in explicitly calculating some of the rational functions arising as Hilbert series.
\begin{theorem}[Residues Theorem]
Let $f:\Cc\to \Cc$ be a meromorphic function and let $\gamma:[a,b]\to \Cc$ be a smooth closed positively oriented curve. Assume that $z_1,\cdots, z_m$ are all the poles of $f$ inside the interior of $\gamma$, and that the index of each one of them with respect to $\gamma$ is 1. Then 
$$\tpi\oint_{\gamma}f(z)dz = \sum_{i=1}^m\Res_{z=z_i}(f).$$
If $z_i$ is a pole of order $a_i$ then the residue at $z_i$ is equal to 
$$\Res_{z=z_i}(f)=\lim_{z\to z_i} \frac{1}{(a_i-1)!}\frac{d^{a_i-1}}{(dz)^{a_i-1}}f(z)(z-z_i)^{a_i}.$$
If the Laurent series of $f$ around $z_i$ is $\sum_{j\in \Z}c_jz^j$ then $$\Res_{z=z_i}f = c_{-1}.$$
\end{theorem}

\end{section}
\begin{section}{The star product on $Zel$}\label{sec:star}
For any finite group $G$, the abelian group $\R(G)$ has an additional structure of a commutative ring. The product is the tensor product of representations over the ground field:
\begin{equation}[V]\star [W] = [V\ot W]\end{equation}
where $g\in G$ acts on $V\ot W$ diagonally: $g\cdot (v\ot w) = gv\ot gw$. For the symmetric group, the structure constants for this multiplication are usually referred to as the
\textit{Kronecker coefficients}. 
For partitions $\la$ and $\mu$ of $n$ we write
\begin{equation} [\S_{\la}]\star [\S_{\mu}]  = \sum_{\nu\parti n} \S_{\nu}^{g(\la,\mu,\nu)}\end{equation} 
The Kronecker coefficients $g(\la,\mu,\nu)$ are much harder to calculate then the Littlewood-Richardson coefficients mentioned above. See for example \cite{IMW} and the introduction in \cite{BVO}. 

Using the isomorphism $Zel\cong \Z[x_1,x_2\ldots ]$, the product $\star: R(S_n)\ot_{\Z} R(S_n)\to R(S_n)$ gives us a new multiplication on $Zel$. This multiplication was not considered in the work of Zelevinsky. We will calculate this multiplication explicitly with respect to the basis of monomials in the indeterminates $x_i$. We begin with the following lemma: 
\begin{lemma} The product $\star$ satisfies the following properties:
\begin{enumerate}
\item The product $\star$ is associative.
\item The product $\star$ is commutative.
\item the product $\star$ is distributive with respect to addition, that is $x\star(y+z) = x\star y + x\star z$. 
\item The map $\star:Zel\ot_{\Z} Zel\to Zel$ is a coalgebra map.
\item The element $x_n\in R(S_n)$ is a unit with respect to $\star$. That is: for every $y\in R(S_n)$ it holds that $y\star x_n = x_n\star y = y$.
\end{enumerate}
\end{lemma}
\begin{proof}
The first three claims are straightforward. The fourth part follows from the fact that taking tensor product of representations commutes with restricting to a subgroup. The last claim is immediate from the fact that $x_n$ represents the trivial representation of $S_n$. 
\end{proof}

The PSH-algebra $Zel$ has a $\Z$-basis given by monomials in the indeterminates $x_n$. 
We now write down the $\star$ product with respect to this basis.
The homogeneous component $\R(S_n)$ has a basis given by all monomials $x_{a_1}x_{a_2}\cdots x_{a_r}$ such that $\sum_i a_i = n$.
We write $\textbf{a}=(a_1,\ldots , a_r)$ and similarly $\textbf{b}=(b_1,\ldots , b_s)$. 
For two such monomials $x_{a_1}\cdots x_{a_r}$ and $x_{b_1}\cdots x_{b_s}$, we consider the following set of matrices with non-negative integer values:
\begin{equation}C_{\textbf{a},\textbf{b}} := \{(c_{ij})\in \M_{r\times s}(\N)| \sum_i c_{ij} = b_j, \sum_j c_{ij} = a_i\}.\end{equation}
For any $c\in C_{\textbf{a},\textbf{b}}$ we write 
$$M(c) = \prod_{i,j} x_{c_{i,j}}\in Zel.$$
We claim the following:
\begin{proposition}\label{prop:Kronecker}
For two tuples $(a_1,\ldots a_r), (b_1,\ldots b_s)$ such that $\sum_i a_i = \sum_j b_j = n$ we have 
$$(x_{a_1}\cdots x_{a_r})\star (x_{b_1}\cdots x_{b_s}) = \sum_{c\in C_{\textbf{a},\textbf{b}}}M(c).$$
\end{proposition}
In order to prove the proposition, we begin by proving the following auxiliary result, following the principle of the Mackey formula:

\begin{lemma} \label{lem:mackey}
	Let $G$ be a finite group, and let $H_1$ and $H_2$ be two subgroups.
	Let $D$ be a set of representatives for the double $H_1-H_2$ cosets in $G$ (i.e. $G=\sqcup_{g\in D} H_1gH_2$). 
	Then there is an isomorphism of $G$-representations 
	$$\Ind_{H_1}^G\one\ot_K \Ind_{H_2}^G\one\cong \bigoplus_{g\in D} \Ind_{H_1\cap gH_2g^{-1}}^G \one.$$
\end{lemma}
\begin{proof}
	We can think of $\Ind_{H_1}^G\one$ as the permutation representation $KG/H_1$, and similarly for $H_2$. 
	As such, $\Ind_{H_1}^G\one\ot_K \Ind_{H_2}^G\one$ is isomorphic to the permutation representation $G/H_1\times G/H_2$.
	It is easy to see that every $G$-orbit contains a unique element of the form $(H_1,gH_2)$ for a unique $g\in D$.
	The stabilizer of $(H_1,gH_2)$ is just $H_1\cap gH_2g^{-1}$. The result now follows easily.
\end{proof}

\begin{proof}[Proof of Proposition \ref{prop:Kronecker}]
Considered as representations of $S_n$, the definition of the multiplication in $Zel$ gives us \begin{equation}x_{a_1}\cdots x_{a_r}= [\Ind_{S_{a_1}\times\cdots\times S_{a_r}}^{S_n}\one]\text{ and } 
x_{b_1}\cdots x_{b_s}=[\Ind_{S_{b_1}\times\cdots\times S_{b_s}}^{S_n}\one].\end{equation} Using the above lemma, we just need to analyze the double $H_1-H_2$ cosets
in $S_n$ and the relevant intersections, where $H_1= S_{a_1}\times\cdots\times S_{a_r}$ and $H_2 = S_{b_1}\times\cdots\times S_{b_s}$.
	
For this, consider the action of $S_n$ on the set \begin{equation}X=\{(X_1,X_2,\ldots X_r)| |X_i| = a_i, \sqcup X_i = \{1,\ldots n\} \}. \end{equation}
The action of $S_n$ on this set is transitive, and the stabilizer of the element $(X_1,\ldots ,X_r)$ with $X_i = \{a_1+a_2+\cdots + a_{i-1} + 1, \cdots, a_1+a_2+\cdots a_{i-1}+a_i\}$ is exactly $H_1$. Similarly, we define the set
\begin{equation}Y = \{(Y_1,Y_2,\ldots Y_s)| |Y_i| = b_i, \sqcup Y_i = \{1,2,\ldots,n\}\}.\end{equation}
The group $S_n$ acts transitively on this set, and the stabilizer of the element $(Y_1,\ldots Y_r)$ where $Y_i = \{b_1+b_2+\cdots b_{i-1}+1,\ldots,b_1+b_2+\cdots + b_i\}$ is exactly $H_2$. 
This easily implies that if $(X_1,\ldots X_r),(X'_1,\ldots X'_r)\in X$ are two tuples which are conjugate by the action of $H_2$ then it holds that
$|X_i\cap Y_j| = |X'_i\cap Y_j|$ for all $i$ and $j$. Indeed, if $g\in H_2$ satisfies $gX_i = X'_i$ then 
$|X'_i\cap Y_j| = |gX_i\cap Y_j| = |X_i\cap g^{-1}Y_j| = |X_i\cap Y_j|$ because $gY_j = Y_j$.
 
On the other hand, the set of cardinalities $\{|X_i\cap Y_j|\}_{i,j}$ already forms a complete set of invariants for the $H_2$-orbit. Indeed, if $(X_1,\ldots X_r)$ and $(X'_1,\ldots X'_r)$	are two tuples in $X$ which satisfy $|X_i\cap Y_j| = |X'_i\cap Y_j|$ for all $i$ and $j$, let $\sigma\in S_n$ be a permutation which satisfies $\sigma(X_i\cap Y_j) = X_i'\cap Y_j$
for every $i$ and every $j$. Since $\{1,\ldots n\}=\sqcup_{i,j}(X_i\cap Y_j)=\sqcup_{i,j}(X_i'\cap Y_j)$ such a permutation exists.
Since $X_i = \sqcup_j (X_i\cap Y_j)$ and $X'_i = \sqcup_j (X'_i\cap Y_j)$ it holds that $\sigma(X_1,\ldots X_r) = (X'_1,\ldots X'_r)$. 
Since $Y_j = \sqcup_i (X_i\cap Y_j) = \sqcup_i (X'_i\cap Y_j)$ it holds that $\sigma(Y_j) = Y_j$, so $\sigma\in H_2$. 

The assignment \begin{equation} H_1\sigma H_2\mapsto (|X_i\cap \sigma(Y_j)|)_{i,j}\end{equation} 
gives a bijection between the set of double $H_1-H_2$ cosets in $S_n$ and $C_{\textbf{a},\textbf{b}}$. For every matrix $(c_{i,j})\in C_{\textbf{a},\textbf{b}}$ let $\sigma_c\in S_n$ be a permutation which belongs to the double $H_1-H_2$-coset corresponding to $c$. 
The intersection $H_1\cap \sigma_c H_2\sigma_c^{-1}$ is then the intersection of the stabilizer of $(X_1,\ldots, X_r)\in X$ and the stabilizer of $(Y_1,\ldots, Y_s)$. This is the same as the stabilizer of the subsets $X_i\cap Y_j$, which is isomorphic to $\prod_{i,j} S_{c_{i,j}}$. 
	
By Lemma \ref{lem:mackey} and by the calculation above we get  
$$(x_{a_1}\cdots x_{a_r})\star(x_{b_1}\cdots x_{b_s}) = [\Ind_{S_{a_1}\times\cdots\times S_{a_r}}^{S_n}\one\ot \Ind_{S_{b_1}\times\cdots\times S_{b_s}}^{S_n}\one] \cong $$
\begin{equation}[\bigoplus_{c\in C_{\textbf{a},\textbf{b}}}\Ind_{\prod_{i,j}S_{c_{i,j}}}^{S_n}\one] =
\sum_{c\in C_{\textbf{a},\textbf{b}}}\prod_{i,j}x_{c_{i,j}}, \text{ so }   \end{equation}
\begin{equation}(x_{a_1}\cdots x_{a_r})\star (x_{b_1}\cdots x_{b_s}) = \sum_{c\in C_{\textbf{a},\textbf{b}}} \prod_{i,j} x_{c_{i,j}}.\end{equation}
This finishes the proof of the proposition. 
\end{proof}
Following the proposition above, we get the following calculation which will be the key in getting concrete results when the dimensions of the relevant vector spaces is 2. 
\begin{proposition}\label{prop:p2} The following equation holds in $Zel$:
	$$\sum_{\lambda\in P_2(n)}[\S_{\lambda}]\star [\S_{\lambda}] = \sum_{2i+j+k=n} x_i^2x_jx_k - \sum_{2i+1+j+k=n}x_ix_{i+1}x_jx_k.$$
\end{proposition}
\begin{proof}
	We use Equation \ref{eq:twoparts} and Proposition \ref{prop:Kronecker}. If $\lambda =  (l_1,l_2)$ with $l_1\geq l_2$ then 
	$$[\S_{\lambda}]\star[\S_{\lambda}] = (x_{l_1}x_{l_2} - x_{l_1+1}x_{l_2-1})\star(x_{l_1}x_{l_2} - x_{l_1+1}x_{l_2-1}) =$$ $$ (x_{l_1}x_{l_2})\star(x_{l_1}x_{l_2}) + (x_{l_1+1}x_{l_2-1})\star(x_{l_1+1}x_{l_2-1})- 
	2(x_{l_1}x_{l_2})\star(x_{l_1+1}x_{l_2-1})= $$
	\begin{equation}\sum_{c\in C_1(\la)}M(c)+  \sum_{c\in C_2(\la)}M(c) - 2 \sum_{c\in C_3(\la)}M(c)\end{equation}
	where \begin{equation}C_1(\la) = \bigg\{\begin{pmatrix} c_{11} & c_{12} \\ c_{21} & c_{22} \end{pmatrix}\bigg| c_{11}+c_{12} = c_{11}+c_{21} = l_1, c_{12}+c_{22}= c_{21} + c_{22} = l_2 \bigg\}, \end{equation}
\begin{equation}C_2(\la) = \bigg\{\begin{pmatrix} c_{11} & c_{12} \\ c_{21} & c_{22} \end{pmatrix}\bigg| c_{11}+c_{12} = c_{11}+c_{21} = l_1+1, c_{12}+c_{22}= c_{21} + c_{22} = l_2-1 \bigg\}, \end{equation}
 \begin{equation} C_3(\la) =  \bigg\{\begin{pmatrix} c_{11} & c_{12} \\ c_{21} & c_{22} \end{pmatrix}\bigg| c_{11}+c_{12}=l_1, c_{11}+c_{21} = l_1+1, c_{12}+c_{22} = l_2-1, c_{21}+c_{22} = l_2\bigg\}.\end{equation}
 Consider the defining properties of $C_1(\la)$. It holds that $c_{12}=c_{21}$ for all matrices there, and that $C_1(\la)\cap C_1(\la') = \varnothing$ when $\la\neq \la'$. 
 By considering all partitions $\la\in P_2(n)$ and recalling that $l_1\geq l_2$ we get 
\begin{equation}\bigsqcup_{\la\in P_2(n)} C_1(\la) = \bigg\{\begin{pmatrix} c_{11} & c_{12} \\ c_{21} & c_{22} \end{pmatrix}\bigg| c_{12}= c_{21}, c_{11}\geq c_{22}, \sum_{i,j}c_{ij}=n\bigg\}.\end{equation}
 Similarly, 
 \begin{equation}\bigsqcup_{\la\in P_2(n)} C_2(\la) = \bigg\{\begin{pmatrix} c_{11} & c_{12} \\ c_{21} & c_{22} \end{pmatrix}\bigg| c_{12}= c_{21}, c_{11}\geq c_{22}+2, \sum_{i,j}c_{ij}=n\bigg\}\end{equation} and 
 \begin{equation}\bigsqcup_{\la\in P_2(n)} C_3(\la) = \bigg\{\begin{pmatrix} c_{11} & c_{12} \\ c_{21} & c_{22} \end{pmatrix}\bigg| c_{12}+1= c_{21}, c_{11}\geq c_{22}+1, \sum_{i,j}c_{ij}=n\bigg\}.\end{equation}
 
 We then have 
 \begin{equation}\sum_{\la\in P_2(n)}\sum_{c\in C_1(\la)}M(c) = \sum_{\substack{2i+j+k=n \\ j\geq k} }x_i^2x_jx_k, \end{equation}
 \begin{equation}\sum_{\la\in P_2(n)}\sum_{c\in C_2(\la)}M(c) = \sum_{\substack{2i+j+k=n\\ j\geq k+2}}x_i^2x_jx_k, \end{equation}
 \begin{equation}\sum_{\la\in P_2(n)}\sum_{c\in C_3(\la)}M(c) = \sum_{\substack{2i+1+j+k=n \\ j\geq k+1}}x_ix_{i+1}x_jx_k, \end{equation}
 where we write $i=c_{12}$, $j = c_{11}$ and $k= c_{22}$. 
 By switching the $j$ and $k$ indices in the $C_2$ summation we get
 $$\sum_{\la\in P_2(n)}\sum_{c\in C_1(\la)}M(c) + \sum_{\la\in P_2(n)}\sum_{c\in C_2(\la)}M(c) = $$
$$ \sum_{\substack{2i+j+k=n \\ j\geq k}} x_i^2x_jx_k +  \sum_{\substack{2i+j+k=n \\ j+2\leq k}}x_i^2x_jx_k = $$
\begin{equation}\sum_{\substack{2i+j+k=n \\ j\neq k-1}} x_i^2x_jx_k = \sum_{2i+j+k=n} x_i^2x_jx_k - \sum_{2i+2k-1=n}x_i^2x_kx_{k-1}\end{equation}
For the $C_3$ sum we use a similar manipulation for changing the indices. We get 
  $$2\sum_{\la\in P_2(n)}\sum_{c\in C_3(\la)}M(c) = \sum_{\la\in P_2(n)}\sum_{c\in C_3(\la)}M(c) + \sum_{\la\in P_2(n)}\sum_{c\in C_3(\la)}M(c) = $$
  $$\sum_{\substack{2i+1+j+k=n \\ j\geq k+1}}x_ix_{i+1}x_jx_k+ \sum_{\substack{2i+1+j+k=n \\ k\geq j+1}}x_ix_{i+1}x_jx_k = $$
  \begin{equation}\sum_{\substack{2i+1+j+k=n \\ j\neq k}}x_ix_{i+1}x_jx_k = \sum_{2i+1+j+k=n}x_ix_{i+1}x_jx_k - 
  \sum_{2i+1+2j=n}x_ix_{i+1}x_j^2. \end{equation}
  Summing it all up, we get 
  $$\sum_{\lambda\in P_2(n)}[\S_{\lambda}]\star [\S_{\lambda}] = $$
   $$\sum_{\la\in P_2(n)}\sum_{c\in C_1(\la)}M(c) + \sum_{\la\in P_2(n)}\sum_{c\in C_2(\la)}M(c) - 2\sum_{\la\in P_2(n)}\sum_{c\in C_3(\la)}M(c) = $$
   $$\sum_{2i+j+k=n} x_i^2x_jx_k - \sum_{2i+2k-1=n}x_i^2x_kx_{k-1} - 
   \bigg(\sum_{2i+1+j+k=n}x_ix_{i+1}x_jx_k - \sum_{2i+1+2j=n}x_ix_{i+1}x_j^2\bigg) = $$
   \begin{equation}\sum_{2i+j+k=n} x_i^2x_jx_k - \sum_{2i+1+j+k=n}x_ix_{i+1}x_jx_k\end{equation}
   where we used the equality \begin{equation}\sum_{2i+2k-1=n}x_i^2x_kx_{k-1} = \sum_{2i+1+2j=n}x_ix_{i+1}x_j^2\end{equation} which follows by relabeling the indices $i,j,k$. This concludes the proof of the proposition. 
\end{proof}

\end{section}

\begin{section}{The invariant ring- the tensor product case}\label{sec:tensorproducts}
In this section we will calculate the Hilbert function of $A:=K[\End(V\ot W)]^{\GL(V)\times\GL(W)}$ in terms of the Kronecker coefficients. We will then calculate the Hilbert series explicitly in case $\dim(V)=\dim(W)=2$. 
We write $U=\End(V\ot W)$ and $\Ga = \GL(V)\times \GL(W)$. 
We write $d_1=\dim(V)$ and $d_2=\dim(W)$
Lemma \ref{lem:inv-coinv} gives us 
$$A_n\cong (((\End(V\ot W)^*)^{\ot n})^{\Ga})_{S_n}\cong 
(((V^*\ot V^{**}\ot W^*\ot W^{**})^{\ot n})^{\Ga})_{S_n}\cong $$
$$(((V\ot V^*)^{\ot n})^{\GL(V)}\ot ((W\ot W^*)^{\ot n})^{\GL(W)})_{S_n}\cong$$
\begin{equation}(((\End(V)^{\ot n})^{\GL(V)})\ot ((\End(W)^{\ot n})^{\Ga}))_{S_n}.\end{equation}
Schur-Weyl duality gives us an isomorphism of $S_n$-representations: 
\begin{equation}(\End(V)^{\ot n})^{\GL(V)}\cong \bigoplus_{\la\in P_{d_1}(n)}\End(\S_{\la}) \text{ and }\end{equation}
\begin{equation}(\End(W^{\ot n}))^{\GL(W)}\cong \bigoplus_{\la\in P_{d_2}(n)}\End(\S_{\mu}).\end{equation}
These isomorphisms are then combined to give us 
$$(A_n)\cong \bigoplus_{\substack{\la\in P_{d_1}(n) \\ \mu\in P_{d_2}(n)}}(\End(\S_{\la})\ot \End(\S_{\mu}))_{S_n}\cong $$ 
\begin{equation}\bigoplus_{\substack{\la\in P_{d_1}(n) \\ \mu\in P_{d_2}(n)}}(\End(\S_{\la}\ot \S_{\mu}))_{S_n}.\end{equation}
Since $S_n$ is a finite group, the natural map $X^{S_n}\to X\to X_{S_n}$ is an isomorphism between the $S_n$-invariants and $S_n$-coinvariants for every $S_n$-representation $X$ (we use here the fact that $K$ has characteristic zero). 
This implies that 
\begin{equation}A_n\cong \bigoplus_{\substack{\la\in P_{d_1}(n) \\ \mu\in P_{d_2}(n)}}\End(\S_{\la}\ot \S_{\mu})^{S_n} =\bigoplus_{\substack{\la\in P_{d_1}(n) \\ \mu\in P_{d_2}(n)}}\End_{S_n}(\S_{\la}\ot \S_{\mu}).\end{equation}

In order to understand the endomorphism ring of the last $S_n$-representation, we can use the decomposition of the tensor product using the Kronecker coefficients $g(\la,\mu,\nu)$. Indeed, we have:
\begin{equation}\S_{\la}\ot \S_{\mu}\cong \bigoplus_{\nu\parti n} \S_{\nu}^{g(\la,\mu,\nu)}\end{equation} and therefore
\begin{equation}\End_{S_n}(\S_{\la}\ot \S_{\mu})\cong \End_{S_n}(\bigoplus_{\nu\parti n} \S_{\nu}^{g(\la,\mu,\nu)})\cong \end{equation}
\begin{equation} \bigoplus_{\nu\parti n} \M_{g(\la,\mu,\nu)}(K).\end{equation} So we have
\begin{equation}A_n\cong \bigoplus_{\substack{\la\in P_{d_1}(n) \\ \mu\in P_{d_2}(n)}} \M_{g(\la,\mu,\nu)}(K)\end{equation}
and we get the following formula for the dimension of $A_n$:
\begin{theorem}\label{thm:maintensorproducts}
We have 
$$\dim(K[\End(V\ot W)]^{\Ga}_n) = \sum_{\substack{\la\in P_{d_1}(n) \\ \mu\in P_{d_2}(n)}} g(\la,\mu,\nu)^2.$$
\end{theorem}

The last formula describes the dimension of $A_n$ using the Kronecker coefficients of $S_n$. 
The Kronecker coefficients, however, are difficult to calculate.
We would like to get a more concrete description for the dimension of $A_n$, at least in case that $\dim(V)$ and $\dim(W)$ are small. In case $\dim(V)=1$ we just get the invariant ring $K[\End(W)]^{\GL(W)}$, which is known to be the polynomial ring on the coefficients of the characteristic polynomial. The case where $\dim(W)=1$ is similar. We will next study the case where $\dim(V)=\dim(W)=2$

\subsection{The case $\dim(V)=\dim(W)=2$}
We use now the fact that all the Specht modules are self-dual (see Section \ref{sec:prelim}). For every two partitions $\la,\mu\parti n$ we thus have
\begin{equation}\End_{S_n}(\S_{\la}\ot \S_{\mu})\cong (\S_{\la}\ot \S_{\mu}\ot \S_{\la}^*\ot \S_{\mu}^*)^{S_n}\cong (\S_{\la}\ot \S_{\la}\ot \S_{\mu}\ot \S_{\mu})^{S_n}.\end{equation}
We rewrite $A_n$ in the following way:
\begin{equation}A_n\cong \bigoplus_{\substack{\la\in P_2(n)\\ \mu\in P_2(n)}} (\S_{\la}\ot \S_{\la}\ot \S_{\mu}\ot \S_{\mu})^{S_n}\cong \end{equation} 
\begin{equation}(\bigoplus_{\la\in P_2(n)} \S_{\la}\ot \S_{\la}\ot\bigoplus_{\mu\in P_2(n)}\S_{\mu}\ot S_{\mu})^{S_n}. \end{equation}

We showed in Proposition \ref{prop:p2} that as an element of $Zel$ the representation $\bigoplus_{\la\in P_2(n)}\S_{\la}\ot\S_{\la}$ is \begin{equation}\sum_{2i+j+k = n} x_i^2x_jx_k - \sum_{2i+1+j+k=n} x_ix_{i+1}x_jx_k.\end{equation}
The dimension of $A_n$ is thus the following inner product:
\begin{equation}\bigg\langle \big(\sum_{2i+j+k=n} x_i^2x_jx_k - \sum_{2i+1+j+k=n}x_ix_{i+1}x_jx_k\big)^{\star 2}, x_n \bigg\rangle \end{equation}
Since for any representation $V$ of $S_n$, $\langle [V],x_n\rangle$ is the dimension of the $S_n$-invariant subspace.
We start with calculating the star product 
$$(\sum_{2i+j+k=n} x_i^2x_jx_k - \sum_{2i+1+j+k=n}x_ix_{i+1}x_jx_k)^{\star 2} =  $$ $$(\sum_{2i+j+k=n} x_i^2x_jx_k - \sum_{2i+1+j+k=n}x_ix_{i+1}x_jx_k)\star $$ \begin{equation}(\sum_{2i+j+k=n} x_i^2x_jx_k - \sum_{2i+1+j+k=n}x_ix_{i+1}x_jx_k).\end{equation}
We use the fact that $\star$ is distributive with respect to the addition.
Let us start with calculating \begin{equation}\sum_{2i+j+k=n}x_i^2x_jx_k\star\sum_{2i+j+k=n}x_i^2x_jx_k.\end{equation}
By Proposition \ref{prop:Kronecker} we get 
\begin{equation} x_{i_1}x_{i_1}x_{j_1}x_{k_1}\star x_{i_2}x_{i_2}x_{j_2}x_{k_2} = \sum_{c\in C_{\textbf{a},\textbf{b}}} M(c)\end{equation} 
where $\textbf{a} = (i_1,i_1,j_1,k_1)$ and $\textbf{b} = (i_2,i_2,j_2,k_2)$.
We are taking here the sum over all the monomials $M(c)$ where $c$ is a $4\times 4$ matrix in which the sums of the rows are $(i_1,i_1,j_1,k_1)$ and the sums of the columns are $(i_2,i_2,j_2,k_2)$. When we take now the product
\begin{equation} (\sum_{2i_1+j_1+k_1=n}x_{i_1}^2x_{j_1}x_{k_1})\star(\sum_{2i_2+j_2+k_2=n}x_{i_2}^2x_{j_2}x_{k_2})\end{equation} 
we get a sum of the form $\sum_c M(c)$ where $c$ now runs through \textit{all} the $4\times 4$ matrices in which the sum of the first row is equal to the sum of the second row and the sum of the first column is equal to the sum of the second column. We introduce the following notations:
for $i,j\in \Z$ and $n\in \N$ we write \begin{equation}C(i,j,n) = \bigg\{ (c_{k,l})\in \M_4(\N)| \sum_l c_{1,l}-c_{2,l} = i, \sum_k c_{k,1}-c_{k,2} = j, \sum_{k,l}c_{k,l} = n\bigg\}\end{equation} 
and \begin{equation} f(i,j,n) = |C(i,j,n)|.\end{equation} 
The above calculation shows us that \begin{equation} \langle (\sum_{2i+j+k=n}x_i^2x_jx_k)^{\star 2},x_n\rangle = f(0,0,n).\end{equation} 
A similar calculation shows that \begin{equation} \langle (\sum_{2i+j+k=n}x_i^2x_jx_k)(\sum_{21+j+k=n}x_{i+1}x_ix_jx_k,x_n\rangle = f(1,0,n) = f(0,1,n)\end{equation}  and \begin{equation} \langle (\sum_{2i+1+j+k=n}x_{i+1}x_ix_jx_k)^{\star 2},x_n\rangle = f(1,1,n).\end{equation} 
This leads us to the following conclusion:

\begin{proposition}
We have $$\dim A_n = f(0,0,n) - f(0,1,n) - f(1,0,n) + f(1,1,n)$$
\end{proposition}
This already gives us a combinatorial description of the Hilbert function of $A$. Next, we will calculate the Hilbert series explicitly.
To do so, we write \begin{equation}F(x) = \sum_{n\geq 0} \dim A_n x^n\in \Z[[x]].\end{equation}
We claim the following:
\begin{theorem}\label{thm:maintensorproductsformula}
The following formula holds in $\Z[[x]]$: 
$$F(x) = \frac{x^{10} - x^8 - x^7 + 2x^6 + 2x^5 + 2x^4 - x^3 - x^2 + 1}{(1-x)(1-x^2)^4(1-x^3)^3(1-x^4)^2}$$
\end{theorem}
The proof of the above theorem will be carried in a few steps. 
We will first consider the commutative ring 
\begin{equation} B=\Q[a^{\pm 1}, b^{\pm 1}][[x]].\end{equation} 
and describe $\sum_{i,j\in \Z,k\in \N}f(i,j,k)a^ib^ix^k$ as the product of reciprocals of linear polynomials in $x$.
We then use this to prove that $\dim A_n$ satisfies a specific recurrence relation. Finally, in the appendix we calculate enough values of $f(i,j,k)$ using \textsc{Mathematica} to conclude the formula for $F(x)$. 

The elements of $B$ can be written as infinite sums of the form \begin{equation} \sum_{i=0}^{\infty} c_ix^i\end{equation}  where $c_i\in \Z[a^{\pm 1}, b^{\pm 1}]$. 
We will write elements of $B$ in a different way, which will be more efficient for the calculations we have here. 
Every element of $B$ can be written uniquely as a sum of the form 
 \begin{equation} \sum_{k=0}^{\infty}\sum_{i,j\in\Z}c_{i,j,k}a^ib^jx^k,\text{ where } c_{i,j,k}\in \Z\end{equation}   where for every $k\geq 0, |\{(i,j)| c_{i,j,k}\neq 0\}|<\infty$.
We will change the order of summation and re-write this element as 
 \begin{equation} \sum_{i,j\in\Z}\sum_{k=0}^{\infty}a^ib^jc_{i,j,k}x^k.\end{equation} 
Writing $g_{i,j} = \sum_{k=0}^{\infty} c_{i,j,k}x^k\in \Z[[x]]$ enables us to write the above element as \begin{equation} \sum_{i,j\in \Z}a^ib^jg_{i,j}.\end{equation} 
Notice that not every collection $(g_{i,j})_{i,j}$ of elements if $\Z[[x]]$ will give us an element of $B$.
Indeed, the elements $g_{i,j}$ should satisfy the following condition: if we write 
$m(g)= min\{i| d_i\neq 0\}$ for $g=\sum_{k=0}^{\infty} d_kx^k$ then it must hold that for every $k\in \N$ the set $\{(i,j)|m(g_{i,j})=k\}$ is finite. We use here the convention that $m(0)=\infty$. It is easy to show that if $(g_{i,j})$ is a collection of elements of $\Z[[x]]$ which satisfy the above condition, then $\sum_{i,j\in\Z}a^ib^jg_{i,j}$ is an element of $B$. We call $g_{i,j}$ the \textit{$(i,j)$-part} of $\sum_{i,j\in \Z}a^ib^jg_{i,j}$.

For every element $c\in \langle a,b\rangle = \{a^ib^j\}_{i,j\in \Z}$ the element $1-cx$ is invertible in $B$. 
The inverse is given explicitly by \begin{equation}\frac{1}{1-cx} = \sum_{k=0}^{\infty}c^kx^k.\end{equation}
We claim the following:
\begin{lemma}\label{lem:gexpansion}
Consider the following element of $B$:
$$g = \frac{1}{(1-ax)^2(1-bx)^2(1-a^{-1}x)^2(1-b^{-1}x)^2(1-abx)}\cdot $$ $$\frac{1}{(1-ab^{-1}x)(1-a^{-1}bx)(1-a^{-1}b^{-1}x)(1-x)^4}.$$
Then $g$ can be written as $$g= \sum_{i,j\in \Z}\sum_{k=0}^{\infty} f(i,j,k)a^ib^jx^k.$$
\end{lemma}
\begin{proof}
Using the above expansion for $\frac{1}{1-cx}$ we get 
$$g = \sum_{k_{11}=0}^{\infty}a^{k_{11}}b^{k_{11}}x^{k_{11}}
\sum_{k_{12}=0}^{\infty}a^{k_{12}}b^{-k_{12}}x^{k_{12}}
\sum_{k_{21}=0}^{\infty}a^{-k_{21}}b^{k_{21}}x^{k_{21}}
\sum_{k_{22}=0}^{\infty}a^{-k_{22}}b^{-k_{22}}x^{k_{22}}$$
$$
\sum_{k_{13}=0}^{\infty}a^{k_{13}}x^{k_{13}}
\sum_{k_{14}=0}^{\infty}a^{k_{14}}x^{k_{14}}
\sum_{k_{23}=0}^{\infty}a^{-k_{23}}x^{k_{23}}
\sum_{k_{24}=0}^{\infty}a^{-k_{24}}x^{k_{24}}$$
$$
\sum_{k_{31}=0}^{\infty}b^{k_{31}}x^{k_{31}}
\sum_{k_{41}=0}^{\infty}b^{k_{41}}x^{k_{41}}
\sum_{k_{32}=0}^{\infty}b^{-k_{32}}x^{k_{32}}
\sum_{k_{42}=0}^{\infty}b^{-k_{42}}x^{k_{42}}$$
$$
\sum_{k_{33}=0}^{\infty}x^{k_{33}}
\sum_{k_{34}=0}^{\infty}x^{k_{34}}
\sum_{k_{43}=0}^{\infty}x^{k_{43}}
\sum_{k_{44}=0}^{\infty}x^{k_{44}}= $$

$$\sum_{k_{ij}}a^{k_{11}+k_{12}+k_{13}+k_{14}-k_{21}-k_{22}-k_{23}-k_{24}}b^{k_{11}+k_{21} + k_{31} + k_{41} - k_{12} - k_{22} - k_{32} - k_{42}}\cdot $$ \begin{equation} x^{k_{11}+k_{12}+k_{13}+k_{14}+ k_{21} + k_{22} + k_{23} + k_{24} + k_{31} + k_{32} + k_{33} + k_{34} + k_{41} + k_{42} + k_{43} + k_{44}}.\end{equation} 
The coefficient of $a^ib^jx^k$ in the above expression will then be the number of all $4\times 4$ matrices $(k_{ij})$ with non-negative integer entries such that the difference between the sum of the elements of the first row and the sum of the elements of the second row is $i$, the difference between the sum of the elements of the first column and the sum of the elements of the second column is $j$, and the overall sum of the matrix entries is $k$. But this is exactly $f(i,j,k)$.
\end{proof}
The next proposition will reduce the calculation of the rational function $F$ to a finite computation which will be carried out in \textsc{Mathematica} in the appendix. We write $f_{ij}\in \Z[[x]]$ for the rational function $\sum_{k=0}^{\infty}f(i,j,k)x^k$. 
\begin{proposition}\label{prop:main}
For $(i,j)\in \{(0,0),(0,1),(1,0),(1,1)\}$ the rational function $$(1-x)^4(1-x^2)^8(1-x^3)^3(1-x^4)^6f_{ij}$$ is a polynomial in $x$ of degree $\leq 37$.
\end{proposition}
\begin{remark} If $i\neq \pm 1$ or $j\neq \pm 1$ we still get a polynomial, which might be of higher degree. We will not use these polynomials here.
\end{remark}
\begin{proof}
For the proof of the proposition, we abbreviate some elements of $B$, and we also introduce a group action on $B$. We write:
\begin{equation} s_i = 1-x^i, i=1,2,3,4\end{equation} 
\begin{equation} t_1:= \frac{1}{1-abx}, t_2:= \frac{1}{1-ab^{-1}x}, t_3 = \frac{1}{1-a^{-1}bx}, t_4:= \frac{1}{1-a^{-1}b^{-1}x}\end{equation} 
\begin{equation} t_5:= \frac{1}{1-ax}, t_6:= \frac{1}{1-a^{-1}x},t_7:= \frac{1}{1-bx}, t_8:= \frac{1}{1-b^{-1}x},
t_9:=\frac{1}{1-x}\end{equation} 
Thus, by Lemma \ref{lem:gexpansion} \begin{equation} g = t_1t_2t_3t_4t_5^2t_6^2t_7^2t_8^2t_9^4.\end{equation}
Since $t_9s_1=1$ we will just consider the element 
\begin{equation}
f = gs_1^4 =  t_1t_2t_3t_4t_5^2t_6^2t_7^2t_8^2\end{equation}
We have an action of the dihedral group of order 8 \begin{equation} \Theta:=\langle \alpha,\beta,\gamma| \alpha^2,\beta^2,\gamma^2,[\beta,\gamma],\alpha\beta\alpha = \gamma\rangle\end{equation}  on the free abelian group of rank 2 generated by $a$ and $b$. This action is given by 
$$\alpha(a) = b, \alpha(b) =a, \beta(a) = a^{-1}, \beta(b) = b, $$ \begin{equation} \gamma(a) = a, \gamma(b) = b^{-1}.\end{equation} 
The action of $\Theta$ on $\langle a,b\rangle$ induces in a natural way an action on $B$. In particular, a direct verification shows that the elements of $\Theta$ permute the elements $\{t_1,t_2,\ldots t_8\}$ and stabilize $t_9$. 
We will use this symmetry to reduce some of the calculation in what follows. For a subset $Z\subseteq B$ we write $\spanG\{Z\}$ for $span_{\Q}\{\theta\cdot z\}_{\theta\in \Theta,z\in Z}$.
We also notice that the element $f\in B$ is stable under the action of $\Theta$. 

During the course of the proof we will use the following equations, which are easy to verify:
\begin{equation}\label{eq:s2} s_2t_1t_4 = t_1+t_4 -1\end{equation}
\begin{equation}s_2t_2t_3 = t_2+t_3 -1\end{equation}
\begin{equation}s_2t_5t_6 = t_5+t_6 -1\end{equation}
\begin{equation}\label{eq:s22}s_2t_7t_8 = t_7+t_8 -1\end{equation}
\begin{equation}\label{eq:s4} s_4t_1t_2t_6 = t_2t_6 + t_1t_6 + t_1t_2(1+a^{-1}x) - t_6 - t_1(1+a^{-1}x) - t_2(1+a^{-1}x) + (1+a^{-1}x) \end{equation}
\begin{equation}\label{eq:s3} s_3t_1t_6t_8 = t_1t_6 + t_1t_8 + t_6t_8 - t_1-t_6-t_8 + 1.\end{equation}
From Equation \ref{eq:s2}-\ref{eq:s22} we get $$s_2^2t_1t_2t_3t_4\in span\{t_1,t_4,1\}\cdot\{t_2,t_3,1\} = 
span\{t_1t_2,t_1t_3,t_1,t_2t_4,t_3t_4,t_4,t_2,t_3,1\} = $$ \begin{equation}\spanG\{t_1t_2,t_1,1\}.\end{equation}
Since $(t_5t_6t_7t_8)^2$ is stable under the action of $\Theta$ we get that 
\begin{equation}s_2^2f\in \spanG\{t_1t_2(t_5t_6t_7t_8)^2,t_1(t_5t_6t_7t_8)^2,(t_5t_6t_7t_8)^2\}.\end{equation}

We continue with analyzing $t_1t_2t_6$ using Equation \ref{eq:s4}. We have
\begin{equation}s_4t_1t_2t_6\in span\{1,a^{-1}x\}\cdot \{t_1t_2,t_1t_6,t_2t_6,t_1,t_2,t_6,1\}.\end{equation}
This implies that $$s_4^2t_1t_2t_6^2 \in span\{1,a^{-1}x\}\cdot \{s_4t_1t_2t_6,s_4t_1t_6^2,s_4t_2t_6^2,s_4t_1t_6,s_4t_2t_6,s_4t_6^2,s_4t_6\} \subseteq $$
$$s_4\cdot span\{1,a^{-1}x\}\cdot \{t_1t_6^2,t_2t_6^2,t_1t_6,t_2t_6,t_6^2,t_6\} \cup \{1,a^{-1}x,a^{-2}x^2\}\cdot\{t_1t_2,t_1t_6,t_2t_6,t_1,t_2,t_6,1\}\subseteq$$
\begin{equation}span\{1,a^{-1},a^{-2}\}\cdot\{1,x,x^2,x^3,x^4,x^5\}\cdot\{t_1t_6^2,t_2t_6^2,t_1t_6,t_2t_6,t_1t_2,t_6^2,t_6,t_1,t_2,1\}\end{equation}
We thus have
$$s_4^2s_2^2f\in \spanG\{s_4^2t_1t_2t_6^2(t_5t_7t_8)^2,s_4^2t_1(t_5t_6t_7t_8)^2,s_4^2(t_5t_6t_7t_8)^2\}\subseteq $$
$$span \{1,x,x^2,x^3,x^4,x^5\}\cdot $$ $$\Big(\spanG\Big(\{1,a^{-1}a^{-2}\}\cdot\{t_1t_6^2,t_2t_6^2,t_1t_6,t_2t_6,t_1t_2,t_6^2,t_6,t_1,t_2,1\}\cdot (t_5t_7t_8)^2\Big) + $$ 
\begin{equation}\spanG\{s_4^2t_1(t_5t_6t_7t_8)^2,s_4^2(t_5t_6t_7t_8)^2\}.\end{equation}
We next use the fact that $s_4^2\in span\{1,x,x^2,x^3,x^4,x^5,x^6,x^7,x^8\}$ and that 
$$\gamma(a^{-1})=a^{-1}, \gamma(t_7t_8) = t_7t_8, \gamma(t_5)=t_5,\gamma(t_6)=t_6$$ to deduce that 
$$s_4^2s_2^2f\in span\{1,x,x^2,x^3,x^4,x^5,x^6,x^7,x^8\}\cdot $$
\begin{equation}\Bigg(\spanG\Big(\{1,a^{-1}a^{-2}\}\cdot\{t_1t_6^2,t_1t_6,t_1t_2,t_6^2,t_6,t_1,t_2,1\}\cdot (t_5t_7t_8)^2\Big)\Bigg).\end{equation}
We rewrite the above equation. We have
$$\{t_1t_6^2,t_1t_6,t_1t_2,t_6^2,t_6,t_1,t_2,1\}\cdot (t_5t_7t_8)^2 = $$
$$\{t_1t_6^2t_8^2(t_5t_7)^2,t_1t_6t_8^2(t_5t_7)^2,t_1t_2t_8^2(t_5t_7)^2,t_6^2t_8^2(t_5t_7)^2,$$ 
\begin{equation}t_6t_8^2(t_5t_7)^2,t_1t_8^2(t_5t_7)^2,t_2t_8^2(t_5t_7)^2,t_8^2(t_5t_7)^2\}\end{equation}
and so 
\begin{equation}\label{eq:s4s2}s_4^2s_2^2f\in span\{x^i\}_{i=0}^8\cdot \Bigg(\spanG\Big((t_5t_7)^2\{1,a^{-1},a^{-2}\}\cdot \end{equation}
$$\{t_1t_6^2t_8^2,t_1t_6t_8^2,t_1t_2t_8^2,t_6^2t_8^2, t_6t_8^2, t_1t_8^2,t_8^2\}\Big)\Bigg).$$ 
We have erased $t_2t_8^2(t_5t_7)^2$ from the spanning set, because $t_1t_8^2t_5^2t_7^2a^i$ and $t_2t_8^2t_5^2t_7^2a^i$ are conjugate under the action of $\gamma\in \Theta$ for every $i$.
From Equation \ref{eq:s3} we deduce that 
\begin{equation}s_3t_1t_6t_8 \in span\{t_1t_6,t_1t_8,t_6t_8,t_1,t_6,t_8,1\}.\end{equation}
This implies that 
$$s_3^2t_1t_6t_8^2\in span\{s_3t_1t_6t_8,s_3t_1t_8^2,s_3t_6t_8^2,s_3t_1t_8,s_3t_6t_8,s_3t_8^2,s_3t_8\}\subseteq $$
$$span\{t_1t_6,t_1t_8,t_6t_8,t_1,t_6,t_8,1, s_3t_1t_8^2,s_3t_6t_8^2,s_3t_1t_8,s_3t_6t_8,s_3t_8^2,s_3t_8\}= $$
$$t_1span\{t_6,t_8,1,s_3t_8^2,s_3t_8\} + span\{t_6t_8,t_6,t_8,1,s_3t_6t_8^2,s_3t_6t_8,s_3t_8^2,s_3t_8\}\subseteq$$ 
\begin{equation}span\{1,x^3\}\cdot\{t_1t_6,t_1t_8,t_1,t_1t_8^2,t_6t_8,t_6,t_8,1,t_6t_8^2,t_8^2\}\end{equation}
and similarly
\begin{equation}s_3^2t_1t_6^2t_8\in span\{t_1t_6,t_1t_8,t_6t_8,t_1,t_6,t_8,1, s_3t_1t_6^2,s_3t_6^2t_8,s_3t_1t_6,s_3t_6t_8,s_3t_6^2,s_3t_6\},\end{equation}
$$s_3^3t_1t_6^2t_8^2\in span\{s_3^2t_1t_6^2t^8,s_3^2t_1t_6t_8^2,s_3^2t_6^2t_8^2,s_3^2t_1t_6t_8,s_3^2t_6^2t_8,s_3^2t_6t_8^2,s_3^2t_6t_8\}\subseteq $$
$$span\{t_1t_6,t_1t_8,t_6t_8,t_1,t_6,t_8,1, s_3t_1t_8^2,s_3t_1t_6^2,s_3t_6t_8^2,s_3t_6^2t_8,$$ $$s_3t_1t_6,s_3t_1t_8,s_3t_6t_8,s_3t_6^2,s_3t_8^2,s_3t_6,s_3t_8,$$ $$s_3^2t_6^2t_8^2,s_3t_1t_6,s_3t_1t_8,s_3t_6t_8,s_3t_1,s_3t_6,s_3t_8,s_3,s_3^2t_6^2t_8,s_3^2t_6t_8^2,s_3^2t_6t_8\}=$$
$$t_1\cdot span\{t_6,t_8,1,s_3t_8^2,s_3t_6^2,s_3t_6,s_3t_8,s_3\} + $$ 
$$span\{t_6t_8,t_6,t_8,1,s_3t_6t_8^2,s_3t_6^2t_8,s_3t_6t_8,s_3t_6^2,s_3t_8^2,s_3t_6,s_3t_8,s_3^2t_6^2t_8^2,s_3^2t_6^2t_8,s_3^2t_6t_8^2,s_3^2t_6t_8\}\subseteq $$
$$span\{1,x^3\}\cdot\Big(t_1\cdot span\{t_6,t_8,1,t_6^2,t_8^2\} +$$ 
\begin{equation}\{t_6t_8,t_6,t_8,1,t_6t_8^2,t_6^2t_8,t_6^2,t_8^2,s_3t_6^2t_8^2,s_3t_6^2t_8,s_3t_6t_8^2,s_3t_6t_8\}\Big).\end{equation}
We want to calculate $s_3^3s_4^2s_2^2f$. For this we use the fact that \begin{equation}s_3\cdot span\{1,x^3\}\subseteq span\{1,x^3,x^6\}.\end{equation} We begin by writing
$$s_3^3\cdot\{t_1t_6^2t_8^2,t_1t_6t_8^2,t_1t_2t_8^2,t_6^2t_8^2, t_6t_8^2, t_1t_8^2,t_8^2\}=$$
$$span\{1,x^3\}\Big(t_1\cdot span\{t_6,t_8,1,t_6^2,t_8^2\} +$$ $$\{t_6t_8,t_6,t_8,1,t_6t_8^2,t_6^2t_8,t_6^2,t_8^2,s_3t_6^2t_8^2,s_3t_6^2t_8,s_3t_6t_8^2,s_3t_6t_8\} + $$
$$s_3\cdot\{t_1t_6,t_1t_8,t_1,t_1t_8^2,t_6t_8,t_6,t_8,1,t_6t_8^2,t_8^2\}\Big)+ $$ 
$$span\{s_3^3t_1t_2t_8^2,s_3^3t_6^2t_8^2,s_3^3t_6t_8^2,s_3^3t_1t_8^2,s_3^3t_8^2\}\subseteq$$
\begin{equation}span\{x^i\}_{i=0}^9 span\{t_1t_6,t_1t_8,t_1,t_1t_6^2,t_1t_8^2,1,t_6,t_8,t_6t_8,t_6t_8^2,t_6^2t_8,t_6^2,t_8^2,t_6^2t_8^2,t_1t_2t_8^2\}.\end{equation}
This implies 
\begin{equation}\label{eq:s3s4s2}s_3^3s_4^2s_2^2f\in span\{x^i\}_{i=0}^{17}\cdot \Bigg(\spanG\Big(\{1,a^{-1},a^{-2}\}\cdot \end{equation}
$$\{t_1t_6(t_5t_7)^2,t_1t_8(t_5t_7)^2,t_1(t_5t_7)^2,t_1t_6^2(t_5t_7)^2,t_1t_8^2(t_5t_7)^2,(t_5t_7)^2,t_6(t_5t_7)^2,$$ $$t_8(t_5t_7)^2,t_6t_8(t_5t_7)^2,t_6t_8^2(t_5t_7)^2,t_6^2t_8(t_5t_7)^2,t_6^2(t_5t_7)^2,t_8^2(t_5t_7)^2,t_6^2t_8^2(t_5t_7)^2,t_1t_2t_8^2(t_5t_7)^2\}\Big)\Bigg).$$
Next, we use Equation \ref{eq:s2} again. This equation implies that 
\begin{equation}s_2t_5t_6\in span\{t_5,t_6,1\}\end{equation}
Therefore, 
$$s_2^3t_5^2t_6^2\in span\{s_2t_5^2,s_2t_5,s_2t_6^2,s_2t_6,s_2, s_2t_5t_6\} \subseteq $$
$$span\{s_2t_5^2,s_2t_5,s_2t_6^2,s_2t_6,s_2, t_5,t_6,1\}\subseteq $$
\begin{equation}span\{1,x^2\}\cdot\{t_5^2,t_5,t_6^2,t_6,1\}.\end{equation}
A similar result holds for the pair $(t_7,t_8)$ using a similar equation.
By going through all the products in Equation \ref{eq:s3s4s2} we get that 
\begin{equation}\label{eq:s2s3s4s2} s_2^6s_3^3s_4^2s_2^2f\in span\{x^i\}_{i=0}^{29}\cdot \spanG\{\{1,a^{-1}a^{-2}\}\cdot t_1^it_5^{j_1}t_6^{j_2}t_7^{k_1}t_8^{k_2}\}_{(i,j_1,j_2,k_1,k_2)\in I}\end{equation} where 
\begin{equation}I = \{(i,j_1,j_2,k_1,k_2) | i\in \{0,1\}, j_1,j_2,k_1,k_2\in \{0,1,2\}, j_1j_2=k_1k_2=ij_2k_2=0\}.\end{equation}
We would like to show that for every $(i,j_1,j_2,k_1,k_2)\in I$ the $(0,0),(0,1),(1,0)$ and $(1,1)$ parts of 
\begin{equation}\spanG\{\{1,a^{-1}a^{-2}\}\cdot t_1^it_5^{j_1}t_6^{j_2}t_7^{k_1}t_8^{k_2}\}\end{equation} are polynomials of low degree.
For this it is enough to show that the $(r_1,r_2)$ parts of $t_1^it_5^{j_1}t_6^{j_2}t_7^{k_1}t_8^{k_2}$ are polynomials of low degree, where $r_1\in \{-1,0,1,2,3\}, r_2\in \{-1,0,1\}$ and $(i,j_1,j_2,k_1,k_2)\in I$. 
We claim the following: 
\begin{lemma}\label{lemma:bounded}
For parameters as above the $(r_1,r_2)$ part of $t_1^it_5^{j_1}t_6^{j_2}t_7^{k_1}t_8^{k_2}$ is a polynomial in $x$ of degree at most 8. 
\end{lemma}
\begin{proof}
We recall the specific description of the elements $t_e\in \Q[a^{\pm 1},b^{\pm 1}][[x]]$:
$$t_1= \frac{1}{1-abx}, t_5 = \frac{1}{1-ax}, t_6 = \frac{1}{1-a^{-1}x}$$
$$t_7 = \frac{1}{1-bx}, t_8 = \frac{1}{1-b^{-1}x}.$$
Writing now $$\frac{1}{1-cx} = \sum_{l\geq 0}c^ix^i$$ for any $c\in \langle a,b\rangle$, 
we get that in all of the products above the following condition holds: one of $a$ or $b$ appears in the product only with positive or only with negative powers. This already limits the possible powers of $x$ which might appear.
We will exhibit this with $t_1t_5t_8$ and with $t_5t_7$. All the other calculations are similar. We have:
$$t_1t_5t_8 = \sum_{l_1,l_2,l_3\geq 0} a^{l_1+l_2}b^{l_1-l_3}x^{l_1+l_2+l_3}.$$
If we write $l_1+l_2=r_1$ and $l_1-l_3 = r_2$ then $l_3 = l_1-r_2 \leq r_1-r_2\leq 4$ and $l_1+l_2\leq 4$, so the highest power in which $x$ will appear in the $(r_1,r_2)$ part  is 4+3=7. 
For $t_5t_7$ we get :
$$t_5t_7 = \sum_{l_1,l_2\geq 0}a^{l_1}b^{-l_2}x^{l_1+l_2}$$ and the powers of $x$ which appear in the $(r_1,r_2)$ parts which are relevant for us are at most $3+1=4$. The other calculations are similar.
\end{proof}
The proof that for any $i,j\in \Z$ the $(i,j)$-part of $s_2^8s_3^3s_4^2f$ is a polynomial (possibly of degree bigger than 37) follows from a similar argument to that which appears in the last lemma. For $(i,j)\in\{(0,0),(0,1),(1,0), (1,1)\}$ the last lemma shows that the we get a polynomial of degree bounded by $8+29=37$. This finishes the proof of Proposition \ref{prop:main}. 
\end{proof}
Proposition \ref{prop:main} proves that $\dim A_n$ satisfies a specific recursive relation. In the appendix we will calculate enough values of $\dim A_n$, and deduce the formula for the rational function in Theorem \ref{thm:maintensorproductsformula}.
\end{section}

\begin{section}{Invariants for a tuple of endomorphisms}\label{sec:tuples}
In this section we will study the ring of invariants $A=K[U]^{\Ga}$ where $U=\End(W)^{\oplus k}$ and $\Ga = \GL(W)$. In other words- these are invariants for $k$-tuple of endomorphisms of $W$ under conjugation by the same automorphism. This question was studied by Procesi in \cite{Procesi}. 
For a tuple $(M_1,\ldots M_k)$ in $\End(W)^{\oplus k}$ Procesi showed that all the $\Ga$-invariants are generated by polynomials of the form $Tr(M_{i_1}M_{i_2}\cdots M_{i_r})$. 
He also described the relations between these polynomials, showing that they all can be deduced from the Cayley-Hamilton Theorem.
We will give here a description of the Hilbert function of the invariant ring.
To do so, we introduce the following iterated Littlewood-Richardson coefficients.
\begin{definition}\label{def:iteratedLR}
Let $n=n_1+n_2+\cdots + n_k$. Let $\lambda_i\parti n_i$ and let $\lambda\parti n$. The iterated Littlewood-Richardson coefficient $c_{(\lambda_i)}^{\lambda}$ is the unique non-negative integer for which the formula $$[\S_{\lambda_1}]\cdot[\S_{\lambda_2}]\cdots[\S_{\lambda_k}] = \sum_{\lambda\parti n}c_{(\lambda_i)}^{\lambda} [\S_{\lambda}]$$  holds in the Algebra $Zel$. 
\end{definition}
Using the associativity of the multiplication in $Zel$ one can easily show that 
\begin{equation}c_{(\lambda_i)}^{\lambda}= \sum_{\mu_1\parti (n_1+n_2)\cdots \mu_{k-2}\parti n-n_k}c_{\lambda_1,\lambda_2}^{\mu_1}c_{\mu_1,\lambda_3}^{\mu_2}\cdots c_{\mu_{k-2},\lambda_k}^{\lambda}.\end{equation}

We write $$U = \End(W)^{\ot k} = \bigoplus_{i=1}^k \End(W)e_i.$$
We will thus think of the tuple $(M_1,\ldots M_k)$ as $\sum_i M_ie_i$. 
We have:
\begin{equation}A_n\cong ((U^{\ot n})^{\Ga})_{S_n} = \big(\bigoplus_{i_1,i_2,\ldots i_n=1}^k (\End(W)^{\ot n}\ot e_{i_1}\ot e_{i_2}\ot\cdots\ot e_{i_n})^{\Ga}\big)_{S_n}.\end{equation}
The direct sum has $k^n$ direct summands, which the group $S_n$ permutes. 
As an $S_n$-set the set of direct summands is the same as the $S_n$-set $\{1,\ldots ,k\}^n$ in which the action is given by 
\begin{equation}\sigma(i_1,\ldots i_n) = (i_{\sigma^{-1}(1)},\ldots i_{\sigma^{-1}(n)}).\end{equation}
The orbits for this action are in one to one correspondence with non-ordered partitions $(n_1,\ldots, n_k)$ of $n$. The orbit which corresponds to $(n_1,\ldots n_k)$ is \begin{equation}\{(i_1,\ldots i_n)| \forall l\leq k |\{j| i_j=l\}|=n_l \}.\end{equation} This orbit contains a unique point $(i_1,i_2,\ldots i_n)$ such that $i_1\leq i_2\leq\cdots\leq i_n$. The stabilizer of this point is the subgroup $S_{n_1}\times\cdots\times S_{n_k}$. Thus, the $S_n$ coinvariants in the description of $A_n$ above is given by 
$$\big(\bigoplus_{i_1,i_2,\ldots, i_n=1}^k (\End(W)^{\ot n}\ot e_{i_1}\ot e_{i_2}\ot\cdots\ot e_{i_n})^{\Ga}\big)_{S_n}\cong $$ $$\bigoplus_{n_1+\cdots +n_k=n}\big( (\End(W)^{\ot n}\ot e_1^{\ot n_1}\ot e_2^{\ot n_2}\ot\cdots\ot e_k^{\ot n_k})^{\Ga}\big)_{S_{n_1}\times\cdots\times S_{n_k}}\cong$$
\begin{equation}\bigoplus_{n_1+\cdots + n_k=n}\big( (\End(W)^{\ot n})^{\Ga}\big)_{S_{n_1}\times\cdots\times S_{n_k}}.\end{equation}
Using Schur-Weyl duality, we get 
\begin{equation}(\End(W)^{\ot n})^{\Ga} \cong \bigoplus_{\la\in P_d(n)}\End(\S_{\lambda})\end{equation}
where $d=\dim(W)$. 
Using again the isomorphism $X_{G}\to X\to X^{G}$ between coinvariants and invariants, this time for the finite group $G=S_{n_1}\times\cdots \times S_{n_k}$, we get 
\begin{equation}\label{eq:Aalgebra} A_n\cong \bigoplus_{n_1+n_2+\cdots +n_k=n}\bigoplus_{\la\in P_d(n)}\End_{S_{n_1}\times\cdots\times S_{n_k}}(\S_{\lambda}).\end{equation}
The restriction of $\S_{\lambda}$ to $S_{n_1}\times\cdots\times S_{n_k}$ is given by 
\begin{equation}\bigoplus_{\la_1\parti n_1}\bigoplus_{\la_2\parti n_2}\cdots \bigoplus_{\la_k\parti n_k}(\S_{\lambda_1}\ot\S_{\lambda_2}\ot\cdots \ot \S_{\lambda_k} )^{\oplus c_{(\lambda_i)}^{\lambda}}.\end{equation}
since $\S_{\la_1}\ot\cdots\ot \S_{\la_k}$ is an irreducible $S_{n_1}\times\cdots\times S_{n_k}$ representation, this implies that 
\begin{equation}A_n\cong \bigoplus_{n_1+n_2+\cdots + n_k= n}\bigoplus_{\la\in P_d(n)}\bigoplus_{\la_1\parti n_1}\cdots\bigoplus_{\la_k\parti n_k}\M_{c_{(\la_i)}^{\la}}(K).\end{equation}
This gives us the following formula for the dimension of $A_n$:
\begin{theorem}\label{thm:maintuples}
For every $n\geq 0$ we have 
$$\dim(A_n) = \sum_{n_1+n_2+\cdots +n_k=n}\sum_{\la\in P_d(n)}\sum_{\la_1\parti n_1}\cdots\sum_{\la_k\parti n_k}(c_{(\la_i)}^{\la})^2$$
\end{theorem}
We finish this section with concrete calculations in case $d=\dim(W)=2$. 

\subsection{The case $\dim(W)=2$}
We will use Frobenius Reciprocity and Lemma \ref{lem:RepsRes} to give a more concrete formula for the dimension of $A_n$. We have 
$$\End_{S_{n_1}\times\cdots \times S_{n_k}}(\S_{\la}) =\Hom_{S_{n_1}\times\cdots \times S_{n_k}}(\S_{\la},\S_{\la})= $$
$$\Hom_{S_{n_1}\times\cdots\times S_{n_k}}(\text{Res}^{S_n}_{S_{n_1}\times\cdots\times S_{n_k}}\S_{\la},\text{Res}^{S_n}_{S_{n_1}\times\cdots\times S_{n_k}}\S_{\la})\cong $$
$$\Hom_{S_n}(\Ind_{S_{n_1}\times\cdots\times S_{n_k}}^{S_n}\text{Res}^{S_n}_{S_{n_1}\times\cdots\times S_{n_k}}\S_{\la},\S_{\la})\cong $$
$$\Hom_{S_n}((KS_n/S_{n_1}\times\cdots\times S_{n_k})\ot \S_{\la},\S_{\la})\cong $$
$$\Hom_{S_n}(\Ind_{S_{n_1}\times\cdots\times S_{n_k}}^{S_n}\one \ot \S_{\la},\S_{\la})\cong $$
\begin{equation}\Hom_{S_n}(\Ind_{S_{n_1}\times\cdots\times S_{n_k}}^{S_n}\one \ot \S_{\la}\ot \S_{\la},\one).\end{equation}
The first isomorphism comes from Frobenius reciprocity, while the second isomorphism comes from Lemma \ref{lem:RepsRes}. 

Proposition \ref{prop:p2} and Equation \ref{eq:Aalgebra} now give us that 
$$\dim(A_n) = \sum_{n_1+n_2+\cdots + n_k=n}\sum_{\la\in P_2(n)}\Big\langle (x_{n_1}x_{n_2}\cdots x_{n_k})\star [\S_{\la}]\star [\S_{\la}],x_n\Big\rangle = $$
$$\sum_{n_1+n_2+\cdots + n_k=n}\Big(\sum_{2i+j+l=n}\big\langle (x_{n_1}\cdots x_{n_k})\star (x_i^2x_jx_l),x_n\big\rangle - $$ 
\begin{equation}\sum_{2i+1+j+l=n}\big\langle (x_{n_1}\cdots x_{n_k})\star (x_ix_{i+1}x_jx_l),x_n\big\rangle\Big)\end{equation}
Using now Proposition \ref{prop:Kronecker} for the multiplication of monomials in $x_i$ under the $\star$-product 
gives us 
\begin{equation}\dim(A_n) = \sum_{n_1+\cdots + n_k=n}\Big(\sum_{2i+j+l=n}|C_{(n_1,\ldots,n_k),(i,i,j,l)}| - \sum_{2i+1+j+l=n} |C_{(n_1,\ldots,n_k),(i+1,i,j,l)}|\Big).\end{equation}
For $i\in \Z$ and $m\in \N$ write now 
\begin{equation}g(i,n) = |\{(c_{r,s})\in \M_{4\times k}(\Z_{\geq 0})|\sum_s c_{1,s}-c_{2,s}=i\}|.\end{equation}
By considering the possible values of the sum of the rows and of the columns of a matrix in the set which appears in the definition of $g(i,n)$ we get the following result:
\begin{proposition}
We have $\dim(A_n) = g(0,n)-g(1,n)$.
\end{proposition}
To get a concrete formula for $\dim(A_n)$ we use again the auxiliary commutative ring $B=\Z[a^{\pm 1},b^{\pm 1}][[x]]$ from the previous section. 
In fact, all the calculations here will take place in the smaller subring $\Z[a^{\pm 1}][[x]]$. 
The proof of the following lemma is completely analogous to the proof of Lemma \ref{lem:gexpansion} and we therefore omit it.
\begin{lemma}
The element $$h = \frac{1}{(1-ax)^k(1-a^{-1}x)^k(1-x)^{2k}}$$ of $B$ has the expansion
$$h = \sum_{i\in \Z,n\geq 0}g(i,n)a^ix^n.$$
\end{lemma}

We thus need to calculate the coefficients of $a^0$ and of $a^1$ for the function $h$.
We shall do so by using the Theorem of Residues from complex analysis. 

For this, consider the function $h$ as a function of the complex variable $a$, and assume that $x$ is a complex number with small modulus. 
The expansion $$\frac{1}{1-ax} = \sum_i a^ix^i$$ is valid when $|a|<1/|x|$ and the expansion
$$\frac{1}{1-a^{-1}x} = \sum_i a^{-i}x^i$$ is valid when $|x|<|a|$. We will assume for the rest of this section that $|x|<0.1$ and that $0.9<|a|<1.1$. 
We then get that the desired rational function we are looking for is given by 
\begin{equation}\tpi\oint_{|a|=1}\frac{(a^{-1}-1)da}{(1-ax)^k(1-a^{-1}x)^k(1-x)^{2k}}\end{equation}
where we have used the fact that the function $g$ is symmetric with respect to inverting $i$, that is $g(i,n) = g(-i,n)$ and in particular $g(-1,n)=g(1,n)$. Since $a=x$ is the only pole of the integrand in the domain $\{a||a|<1\}\subseteq \mathbb{C}$ we get 
$$\tpi\oint_{|a|=1}\frac{(a^{-1}-1)da}{(1-ax)^k(1-a^{-1}x)^k(1-x)^{2k}}= \text{Res}_{a=x}\frac{(a^{-1}-1)}{(1-ax)^k(1-a^{-1}x)^k(1-x)^{2k}}= $$
\begin{equation}\frac{1}{(1-x)^{2k}}\text{Res}_{a=x}\frac{a^{-1}-1}{(1-ax)^k(1-a^{-1}x)^k}\end{equation}
In order to find the residue we write the Laurent series of the function around $a=x$, and we use the equality
\begin{equation}\frac{1}{(1-w)^n} = \sum_{i=0}^{\infty} \binom{i+n-1}{n-1}w^i\text{ for } w\in \Cc \text{ with } |w|<1. \end{equation}

We introduce the variable $z=a-x$. 
By substitute we get 
$$ \frac{1}{(1-x)^{2k}}\cdot \frac{a^{-1}-1}{(1-ax)^k(1-a^{-1}x)^k} = \frac{a^{k-1}-a^k}{(1-x)^{2k}(1-ax)^k(a-x)^k} =$$ $$\frac{(z+x)^{k-1}-(z+x)^k}{(1-x)^{2k}(1-(z+x)x)^k z^k} = 
\frac{(z+x)^{k-1}-(z+x)^k}{(1-x)^{2k}(1-x^2-xz)^k z^k} =$$ $$\frac{1}{(1-x^2)^k(1-x)^{2k}}\frac{(z+x)^{k-1}-(z+x)^k}{(1-\frac{x}{1-x^2}z)^k z^k} = $$
\vspace{15 pt}
$$z^{-k}\frac{1}{(1-x^2)^k(1-x)^{2k}}( \sum_{i=0}^{k-1}\binom{k-1}{i}z^ix^{k-1-i}-\sum_{i=0}^k \binom{k}{i}z^ix^{k-i})\sum_{j=0}^{\infty}\binom{j+k-1}{k-1}\frac{x^j}{(1-x^2)^j}z^j= $$
\vspace{15 pt}
$$z^{-k}\frac{1}{(1-x^2)^k(1-x)^{2k}}\Bigg[\sum_{i=0}^{k-1}\sum_{j=0}^{\infty} \binom{k-1}{i}\binom{j+k-1}{k-1}\frac{x^{k+j-1-i}}{(1-x^2)^j}z^{i+j}-$$ 
\vspace{15 pt}
\begin{equation}\sum_{i=0}^k\sum_{j=0}^{\infty}\binom{k}{i}\binom{j+k-1}{k-1}\frac{x^{k+j-i}}{(1-x^2)^j}z^{i+j}\Bigg]\end{equation} \vspace{15 pt}
The residue we are looking for is the coefficient of $z^{-1}$ in the above expression. It is equal to 
$$\frac{1}{(1-x^2)^k(1-x)^{2k}}\cdot \Bigg[\sum_{i=0}^{k-1} \binom{k-1}{i}\binom{2k-2-i}{k-1}\frac{x^{2k-2-2i}}{(1-x^2)^{k-1-i}}-$$ \begin{equation}\binom{k}{i}\binom{2k-2-i}{k-1}\frac{x^{2k-1-2i}}{(1-x^2)^{k-1-i}}\Bigg].\end{equation} 
This finishes the proof of Theorem \ref{thm:maintuples2}.
For small values of $k$ we get the following explicit formulas:\\
For $k=1$ we get \begin{equation}\frac{1}{(1-x^2)(1-x)^2}(1-x) = \frac{1}{(1-x^2)(1-x)}.\end{equation}
This is consistent with the fact that in this case the invariant ring is a polynomial ring in the variables $Tr(M_1)$ and $Tr(M_1^2)$.\\
For $k=2$ we get 
$$\frac{1}{(1-x^2)^2(1-x)^4}\Big[2\frac{x^2}{1-x^2} + \frac{x^0}{(1-x^2)^0} -  2\frac{x^3}{1-x^2} - 2\frac{x}{(1-x^2)^0} \Big]=$$
$$\frac{1}{(1-x^2)^3(1-x)^4}\Big[2x^2 + 1 - x^2 - 2x^3 - 2x + 2x^3 \Big] = $$ \begin{equation}\frac{1}{(1-x^2)^3(1-x)^4}(x^2-2x-1) = \frac{1}{(1-x^2)^3(1-x)^2}.\end{equation}
This is consistent with the fact that in this case the invariant ring is a polynomial ring in $Tr(M_1), Tr(M_1^2),Tr(M_2), Tr(M_2^2)$ and $Tr(M_1M_2)$.

For $k=3$ we get 
$$\frac{1}{(1-x^2)^3(1-x)^6}\Big[ 6\frac{x^4}{(1-x^2)^2} + 6\frac{x^2}{1-x^2} + \frac{x^0}{(1-x^2)^{0}} - $$ $$6\frac{x^5}{(1-x^2)^2} - 9\frac{x^3}{(1-x^2)^1} - 3\frac{x^{1}}{(1-x^2)^0}\Big] = $$
$$\frac{1}{(1-x^2)^5(1-x)^6}\Big[6x^4 + 6x^2 - 6x^4 + 1-2x^2 + x^4 - 6x^5 - 9x^3 + 9x^5 - 3x + 6x^3 - 3x^5\Big]=$$
$$\frac{1}{(1-x^2)^5(1-x)^6}\Big[ x^4 - 3x^3 + 4x^2 - 3x + 1\Big] = \frac{(1-x)^2(1-x+x^2)}{(1-x^2)^5(1-x)^6}=$$
\begin{equation}\frac{1-x+x^2}{(1-x^2)^5(1-x)^4}.\end{equation}

\end{section}
\appendix
\begin{section}{Calculation of the first 100 values of $f(i,j,n)$ and $F(x)$ using \textsc{Mathematica}}
\begin{center}
\textsc{by DEJAN GOVC}
\end{center}
In this appendix we will calculate the values of $f(i,j,n)$ for $|i|,|j|,n\leq 100$ for the function $f(i,j,n)$ introduced in Section \ref{sec:tensorproducts}. Notice that $f(i,j,n)=0$ if $|i|>n$ or $|j|>n$. We begin by introducing some auxiliary functions following Lemma \ref{lem:gexpansion}:
$$f^1(i,j,n) = \bigg|\bigg\{ (c_{k,l})\in \M_4(\Z_{\geq 0})| \sum_l c_{1,l}-c_{2,l} = i, \sum_k c_{k,1}-c_{k,2} = j, \sum_{k,l}c_{k,l} = m, $$\begin{equation}c_{i,j}=0\text{ unless } (i,j)=(1,1)\}\bigg\}\bigg|
\end{equation}
$$f^2(i,j,n) = \bigg|\bigg\{ (c_{k,l})\in \M_4(\Z_{\geq 0})| \sum_l c_{1,l}-c_{2,l} = i, \sum_k c_{k,1}-c_{k,2} = j, \sum_{k,l}c_{k,l} = m, $$\begin{equation}c_{i,j}=0\text{ unless } (i,j)=(1,2)\}\bigg\}\bigg|
\end{equation}
$$f^3(i,j,n) = \bigg|\bigg\{ (c_{k,l})\in \M_4(\Z_{\geq 0})| \sum_l c_{1,l}-c_{2,l} = i, \sum_k c_{k,1}-c_{k,2} = j, \sum_{k,l}c_{k,l} = m, $$\begin{equation}c_{i,j}=0\text{ unless } (i,j)=(2,1)\}\bigg\}\bigg|
\end{equation}
$$f^4(i,j,n) = \bigg|\bigg\{ (c_{k,l})\in \M_4(\Z_{\geq 0})| \sum_l c_{1,l}-c_{2,l} = i, \sum_k c_{k,1}-c_{k,2} = j, \sum_{k,l}c_{k,l} = m, $$\begin{equation}c_{i,j}=0\text{ unless } (i,j)=(2,2)\}\bigg\}\bigg|
\end{equation}
$$f^5(i,j,n) = \bigg|\bigg\{ (c_{k,l})\in \M_4(\Z_{\geq 0})| \sum_l c_{1,l}-c_{2,l} = i, \sum_k c_{k,1}-c_{k,2} = j, \sum_{k,l}c_{k,l} = m, $$\begin{equation}c_{i,j}=0\text{ unless } (i,j)\in \{(1,3),(1,4),(2,3),(2,4)\}\bigg\}\bigg|
\end{equation}
$$f^6(i,j,n) = \bigg|\bigg\{ (c_{k,l})\in \M_4(\Z_{\geq 0})| \sum_l c_{1,l}-c_{2,l} = i, \sum_k c_{k,1}-c_{k,2} = j, \sum_{k,l}c_{k,l} = m, $$\begin{equation}c_{i,j}=0\text{ unless } (i,j)\in \{(3,1),(3,2),(4,1),(4,2)\}\bigg\}\bigg|
\end{equation}
$$f^7(i,j,n) = \bigg|\bigg\{ (c_{k,l})\in \M_4(\Z_{\geq 0})| \sum_l c_{1,l}-c_{2,l} = i, \sum_k c_{k,1}-c_{k,2} = j, \sum_{k,l}c_{k,l} = m, $$\begin{equation}c_{i,j}=0\text{ unless } (i,j)\in \{(3,3),(3,4),(4,3),(4,4)\}\bigg\}\bigg|
\end{equation}
In other words, each of the functions $f^l(i,j,n)$ counts the number of matrices with the same defining property of $f(i,j,n)$, under the additional restrictions that only a limited subset of the entries are non-zero. The following interpolation formula is immediate:
\begin{lemma} We have 
$$f(i,j,n) = \sum f^1(i_1,j_1,n_1)f^2(i_2,j_2,n_2)f^3(i_3,j_3,n_3)f^4(i_4,j_4,n_4)\cdot $$ $$f^5(i_5,j_5,n_5)f^6(i_6,j_6,n_6)f^7(i_7,j_7,n_7)$$
where the sum is taken over all $n_r,i_r,j_r$ such that $\sum_r n_r = n, \sum_r i_r=i$ and $\sum_r j_r=j$.
\end{lemma}
We also have the following formulas for the different $f^r(i,j,n)$ functions:
$$f^1(i,j,n) = \delta_{n,i}\delta_{n,j}, f^2(i,j,n) = \delta_{i,n}\delta_{-j,n}$$
$$f^3(i,j,n) = \delta_{-n,i}\delta_{n,j}, f^4(i,j,n) = \delta_{-i,n}\delta_{-j,n}$$
$$f^5(i,j,n) = \delta_{j,0}\delta_{n-i\text{ mod }2, 0}((n-i)/2+1)((n+i)/2+1)$$
$$f^6(i,j,n) = \delta_{i,0}\delta_{n-j\text{ mod }2, 0}((n-j)/2+1)((n+j)/2+1)$$
\begin{equation}f^7(i,j,n) = \binom{n+3}{3}\end{equation}
The interpolation formula above enables us to calculate the values of $f(i,j,n)$ using \textsc{Mathematica}. We wrote the following code:
\begin{verbatim}
 ClearAll[F,f,dir,delta]
 dir={{1,1,1},{1,-1,1},{-1,1,1},{-1,-1,1}};
 F[k_,{i_,j_,n_}]:=F[k,{i,j,n}]=Sum[F[k-1,{i,j,n}-r dir[[k]]],{r,0,n+1}]
\end{verbatim}
This gives interpolation with the functions $f_1,f_2,f_3,f_4$.
\begin{verbatim}
 F[0,{i_,j_,n_}]:=F[0,{i,j,n}]=Sum[If[EvenQ[n-i-j-m],
 Sum[Binomial[m+3,3](k+1)(k+Abs[i]+1)((n-Abs[i]-Abs[j]-m)/2-k+1)
 ((n-Abs[i]-Abs[j]-m)/2-k+Abs[j]+1),
 {k,0,(n-Abs[i]-Abs[j]-m)/2}],0],{m,0,n-Abs[i]-Abs[j]}]
\end{verbatim}
This part gives interpolation with the functions $f_5,f_6, f_7$.
\begin{verbatim}
 f[n_]:=f[n]=F[4,{0,0,n}]-2F[4,{1,0,n}]+F[4,{1,1,n}]
\end{verbatim}
Here we restrict our attention to the relevant alternating sum.
\begin{verbatim}
 Table[f[n],{n,0,100}]//TableForm
\end{verbatim}
This gives the following output:
\begin{verbatim}
 1,1,4,6,16,23,52,77,150,224,396,583,964,1395,2180,3100,4639,6466,9344,12785,
 17936,24121,33008,43674,58512,76277,100312,129009,166932,212022,270448,339605,
 427677,531462,661652,814348,1003396,1224088,1494124,1807954,2187942,2627594,
 3154972,3762544,4485172,5314292,6292836,7411150,8721791,10213967,11951528,
 13922650,16204356,18783815,21753488,25099607,28932476,33237650,38145976,
 43642527,49881864,56848831,64725080,73495746,83373309,94343640,106654388,
 120292717,135546036,152403681,171197884,191920988,214955830,240298735,
 268389268,299229137,333321320,370674266,411861940,456901107,506444699,
 560519876,619867224,684526384,755335320,832348504,916511528,1007896684,
 1107568268,1215619404,1333245416,1460563640,1598913368,1748440272,1910641560,
 2085695460,2275272477,2479588053,2700502140,2938272966,3194967240
\end{verbatim}
Proposition \ref{prop:main} gives us a recursive formula for $f(0,0,n)-f(0,1,n)-f(1,0,n)+f(1,1,n)$ and we get the following explicit rational function:
\begin{verbatim}
 delta[S_,f_]:=delta[S,f]=delta[Most[S],(f[#]-f[#-Last[S]])&]
 delta[{},f_]:=delta[{},f]=f
\end{verbatim}
Here we multiply $f$ by the polynomial $(1-x^4)^2(1-x^3)^3(1-x^2)^8(1-x)^4$ as suggested by Proposition \ref{prop:main}.
\begin{verbatim}
 g[n_]:=g[n]=delta[{4,4,3,3,3,2,2,2,2,2,2,2,2,1,1,1,1},f][n]
 Factor[Plus@@Table[g[n]x^n,{n,0,100}]]
\end{verbatim}
This yields the following output:
\begin{verbatim}
 -(-1+x)^7 (1+x)^4 (1-x^2-x^3+2 x^4+2 x^5+2 x^6-x^7-x^8+x^10)
\end{verbatim}
After canceling common denominators, we get Theorem \ref{thm:maintensorproductsformula}.
\end{section}


\begin{thebibliography}{abc}
\bibitem[Ah79]{Ahlfors} L. Ahlfors, Complex Analysis : An Introduction to The Theory of Analytic Functions of One Complex Variable, second edition, McGraw-Hill (1979) 
\bibitem[AM69]{AM} M. Atiyah, and I. MacDonald, Introduction to commutative algebra, Addison-Wesley-Longman (1969).
\bibitem[BVO15]{BVO} C. Bowmann, M. De Visscher and R. Orellana, The partition algebra and the Kronecker coefficients, Transactions of the American Mathematical Society, Volume 367, Number 5, Pages 3647-3667 (2015).
\bibitem[DKS03]{DKS} S. Datt, V. Kodiyalam and V.S. Sunder, Complete invariants for complex semisimple Hopf algebras, Math. Res. Lett. 10 (5-6) 571-586 (2003).
\bibitem[IMW17]{IMW} C. Ikenmeyer, K. D. Mulmuley and M. Walter, On vanishing of Kronecker coefficients, Computational Complexity, Volume 26, Issue 4, pp 949-992 (2017).
\bibitem[Me16]{meir3} E. Meir, Descent, fields of invariants, and generic forms via symmetric monoidal categories, Journal of Pure and Applied Algebra 220 2077-2011 (2016).
\bibitem[Me17]{meir1} E. Meir, Semisimple Hopf algebras via geometric invariant theory, Advances in Mathematics 311, 61-90 (2017).
\bibitem[Me19]{meir2} E. Meir, Hopf cocycle deformations and invariant theory, Math. Z. DOI 10.1007/s00209-019-02326-5 (2019)
\bibitem[Pr76]{Procesi} C. Procesi, The invariant theory of 
n x n matrices, Adv. Math. 19 306-381 (1976).
\bibitem[Sa01]{Sagan} B. E. Sagan, The Symmetric Group: Representations, Combinatorial Algorithms, and Symmetric Functions. Graduate Texts in Mathematics, Vol. 203, Springer (2001).
\bibitem[Ze81]{Zelevinsky} A. Zelevinsky, Representations of Finite Classical Groups: A Hopf Algebra Approach, Springer Lecture Notes in Mathematics (1981). 
 \end{thebibliography}
\end{document}